\documentclass[11pt,a4paper]{article}

\usepackage[T1]{fontenc}
\usepackage{lmodern}
\usepackage{microtype}
\usepackage{geometry}
\usepackage{amsmath,amssymb,amsthm,mathtools}
\usepackage{enumitem}
\usepackage{authblk}
\usepackage[hidelinks]{hyperref}
\usepackage[nameinlink,noabbrev]{cleveref}

\numberwithin{equation}{section}
\newtheorem{theorem}{Theorem}[section]
\newtheorem{proposition}[theorem]{Proposition}
\newtheorem{lemma}[theorem]{Lemma}
\newtheorem{assumption}[theorem]{Assumption}
\theoremstyle{definition}
\newtheorem{definition}[theorem]{Definition}
\theoremstyle{remark}
\newtheorem{remark}[theorem]{Remark}

\newcommand{\R}{\mathbb R}
\newcommand{\C}{\mathbb C}
\newcommand{\supp}{\operatorname{supp}}
\newcommand{\cH}{\mathcal H}
\newcommand{\cD}{\mathcal D}
\newcommand{\cL}{\mathcal L}
\newcommand{\cP}{\mathcal P}
\newcommand{\cA}{\mathcal A}
\newcommand{\cV}{\mathcal V}
\newcommand{\cX}{\mathcal X}
\newcommand{\cY}{\mathcal Y}

\title{Time-Dependent Potential Recovery from Local Data on\\
Conformally Transversally Anisotropic Manifolds}
\author[1]{Yujian Zheng\thanks{Email: \href{mailto:d202280014@hust.edu.cn}{d202280014@hust.edu.cn}}}
\author[1,2]{Zhiwen Duan\thanks{Email: \href{mailto:duanzhw@hust.edu.cn}{duanzhw@hust.edu.cn}}}
\author[1]{Shiqi Jing\thanks{Corresponding author. Email: \href{mailto:d202580010@hust.edu.cn}{d202580010@hust.edu.cn}}}
\affil[1]{School of Mathematics and Statistics, Huazhong University of Science and Technology, Wuhan 430074, Hubei, China}
\affil[2]{Hubei Key Laboratory of Engineering Modeling and Scientific Computing, Huazhong University of Science and Technology, Wuhan 430074, Hubei, China}
\date{}

\begin{document}
\maketitle

\begin{abstract}
We study the recovery of a time-dependent potential in a parabolic equation on a
conformally transversally anisotropic manifold.  The initial state and a lateral
Dirichlet datum supported near the front face are used as inputs, while the
terminal state and the Neumann trace near the back face are observed.  Assuming
injectivity of the attenuated geodesic ray transform on the transversal manifold
for all sufficiently small constant attenuations, we prove that these partial
input--output data uniquely determine the potential.  No simplicity assumption
is imposed on the transversal manifold.  The proof combines a
conformal reduction, transversal Gaussian beam quasimodes, boundary Carleman
estimates, and geometric optics solutions with prescribed lateral supports.
\end{abstract}



\section{Introduction and statement of the main result}
\label{sec:introduction}

Let \(0<T<\infty\), and let \((M,g)\) be a smooth, compact, oriented Riemannian
manifold of dimension \(n\geq3\) with smooth boundary.  We set
\[
 Q=(0,T)\times M^{\mathrm{int}},\qquad
 \overline Q=[0,T]\times M,\qquad
 \Sigma=(0,T)\times\partial M.
\]
The Laplace--Beltrami operator associated with \(g\) is denoted by \(\Delta_g\).
In local coordinates,
\[
 \Delta_g=|g|^{-1/2}\partial_{x_j}
 \bigl(|g|^{1/2}g^{jk}\partial_{x_k}\bigr),
\]
where \((g^{jk})\) is the inverse of \((g_{jk})\),
\(|g|=\det(g_{jk})\), and the Einstein summation convention is used.

We assume that \((M,g)\) admits a conformally transversally anisotropic
realization.  The conformal factor in this realization is denoted by \(c\).  For
\(q\in C(\overline Q;\C)\), consider the initial-boundary value problem
\begin{equation}
 \begin{cases}
  (c(x)^{-1}\partial_t-\Delta_g+q(t,x))u=0&\text{in }Q,\\
  u=f&\text{on }\Sigma,\\
  u(0,\cdot)=u_0&\text{in }M.
 \end{cases}
 \label{eq:original-ibvp}
\end{equation}
Our objective is to recover \(q\) from partial lateral boundary measurements,
with the initial state regarded as an input and the terminal state as an output.
Such a potential may represent a spatially and temporally varying reaction or
absorption rate in heat conduction, mass transport, and related diffusion
processes.

\begin{definition}\label{def:cta}
A smooth compact Riemannian manifold \((M,g)\) of dimension \(n\geq3\) is called
\emph{conformally transversally anisotropic} (CTA) if
\begin{equation}
 M\Subset\R\times M_0^{\mathrm{int}},\qquad
 g=c\,(dx_1^2\oplus g_0),
 \label{eq:cta-geometry}
\end{equation}
where \((M_0,g_0)\) is a smooth compact \((n-1)\)-dimensional Riemannian
manifold with smooth boundary and \(c\in C^\infty(\R\times M_0)\) is strictly
positive.  The manifold \((M_0,g_0)\) is called the transversal manifold.
\end{definition}

The coefficient \(c\) in \eqref{eq:original-ibvp} is the restriction to \(M\)
of the conformal factor in \eqref{eq:cta-geometry}; it is not an additional
coefficient independent of the geometry.  The factor \(c^{-1}\) multiplying
\(\partial_t\) in \eqref{eq:original-ibvp} is deliberate: with this
normalization, conjugation by the appropriate power of the known conformal
factor reduces the equation exactly to a parabolic equation for the product
metric, changing only the zeroth-order potential; see
\cref{sec:preliminaries}.  Examples of CTA manifolds include
precompact smooth proper subsets of Euclidean, spherical, and hyperbolic spaces;
see \cite{DosSantosFerreiraEtAl2016}.  We write
\[
 x=(x_1,x')\in\R\times M_0
\]
and reserve the notation
\begin{equation}
 \phi(x):=x_1
 \label{eq:spatial-lcw}
\end{equation}
for the spatial limiting Carleman weight.  No simplicity assumption is imposed
on \((M_0,g_0)\).  Instead, we assume injectivity of a family of attenuated
geodesic ray transforms.

A unit-speed geodesic segment \(\gamma:[0,L]\to M_0\) is called
\emph{non-tangential} if
\[
 \gamma(0),\gamma(L)\in\partial M_0,\qquad
 \gamma((0,L))\subset M_0^{\mathrm{int}},
\]
and both endpoint velocities are non-tangential to \(\partial M_0\).  For
\(a\in\R\) and \(F\in C(M_0;\C)\), define
\begin{equation}
 I_aF(\gamma):=\int_0^L e^{ar}F(\gamma(r))\,dr.
 \label{eq:attenuated-ray-transform}
\end{equation}

\begin{assumption}\label{ass:ray-injectivity}
There exists \(\varepsilon_0>0\) such that, for every \(a\in\R\) with
\(|a|<\varepsilon_0\) and every \(F\in C(M_0;\C)\), the identities
\(I_aF(\gamma)=0\) for all non-tangential unit-speed geodesics \(\gamma\) imply
\(F=0\) in \(M_0\).
\end{assumption}

Assumption~\ref{ass:ray-injectivity} is satisfied, for example, when
\((M_0,g_0)\) is a simple surface, since injectivity is then known for arbitrary
smooth scalar attenuations; see \cite{SaloUhlmann2011}.

Let \(\nu\) be the outward unit normal to \(\partial M\) with respect to \(g\),
and define the front and back faces by
\begin{equation}
 \partial M_+:=\{x\in\partial M:\partial_\nu\phi(x)\geq0\},
 \qquad
 \partial M_-:=\{x\in\partial M:\partial_\nu\phi(x)\leq0\}.
 \label{eq:front-back-faces}
\end{equation}
Let \(U_0,V_0\subset\partial M\) be open neighborhoods of \(\partial M_+\) and
\(\partial M_-\), respectively, and set
\[
 U=(0,T)\times U_0,\qquad V=(0,T)\times V_0.
\]

For \(r,s>0\), let
\[
 H^{r,s}(\Sigma)
 :=L^2(0,T;H^r(\partial M))\cap H^s(0,T;L^2(\partial M)).
\]
We denote the closure of \(C_0^\infty(\Sigma)\) in this space by
\(H_0^{r,s}(\Sigma)\), and set
\[
 H^{-r,-s}(\Sigma):=\bigl(H_0^{r,s}(\Sigma)\bigr)'.
\]
Also, \(H^{-1}(M)=(H_0^1(M))'\).  Unless stated otherwise, the support of a
generalized lateral trace is its distributional support relative to \(\Sigma\),
and closures of subsets of \(\Sigma\) are taken relative to \(\Sigma\).

For later use, we fix the notation for the lateral and temporal traces.  If
\(u\in C^\infty(\overline Q)\), set
\begin{equation}
 \tau_0u:=u|_\Sigma,\qquad
 \tau_1u:=\partial_\nu u|_\Sigma,\qquad
 r_0u:=u(0,\cdot),\qquad
 r_Tu:=u(T,\cdot).
 \label{eq:four-traces}
\end{equation}
Appendix~\ref{app:well-posedness} shows that these maps admit generalized
extensions to the graph spaces relevant to the forward and backward parabolic
operators.  In particular, for the transposition solutions used below,
\[
 \tau_0u\in H^{-1/2,-1/4}(\Sigma),\qquad
 \tau_1u\in H^{-3/2,-3/4}(\Sigma),\qquad
 r_0u,r_Tu\in H^{-1}(M).
\]
We use the same symbols for the generalized traces.  Thus the boundary and
initial conditions in \eqref{eq:original-ibvp} are understood as
\(\tau_0u=f\) and \(r_0u=u_0\).

Define
\begin{equation}
 \cD_U:=\{f\in H^{-1/2,-1/4}(\Sigma):
             \supp f\subset\overline U\},
 \label{eq:admissible-dirichlet-data}
\end{equation}
with the inherited norm.  The Neumann output is taken in the restriction space
\begin{equation}
 H^{-3/2,-3/4}(V)
 :=\{F|_V:F\in H^{-3/2,-3/4}(\Sigma)\},
 \label{eq:restriction-space}
\end{equation}
equipped with the quotient norm.

For each \((u_0,f)\in H^{-1}(M)\times\cD_U\), let \(u\in L^2(Q)\) be the
transposition solution of \eqref{eq:original-ibvp}, whose existence and
uniqueness are proved in Appendix~\ref{app:well-posedness}.  We define
\begin{equation}
 \begin{aligned}
 \Lambda_{c,g,q}^{U,V}:H^{-1}(M)\times\cD_U
 &\longrightarrow H^{-1}(M)\times H^{-3/2,-3/4}(V),\\
 (u_0,f)&\longmapsto\bigl(r_Tu,\tau_1u|_V\bigr).
 \end{aligned}
 \label{eq:partial-map}
\end{equation}
Thus the input consists of an initial state and a lateral Dirichlet datum
supported in \(U\), while the output consists of the terminal state and the
Neumann trace measured on \(V\).  The boundedness of
\(\Lambda_{c,g,q}^{U,V}\) is also established in
Appendix~\ref{app:well-posedness}.

\begin{theorem}\label{thm:main}
Let \((M,g)\) be a CTA manifold of dimension \(n\geq3\), and suppose that
\cref{ass:ray-injectivity} holds for \((M_0,g_0)\).  Let
\(q_1,q_2\in C(\overline Q;\C)\) satisfy \(q_1=q_2\) on \(\Sigma\).  If
\[
 \Lambda_{c,g,q_1}^{U,V}=\Lambda_{c,g,q_2}^{U,V},
\]
then \(q_1=q_2\) in \(Q\).
\end{theorem}

\begin{remark}\label{rem:boundary-agreement}
The condition \(q_1=q_2\) on \(\Sigma\) is an a priori boundary agreement
assumption.  It is used only in the final step of the proof to ensure that the
zero extension of \(q_2-q_1\) from \(M\) to \(\R\times M_0\) is continuous.
The Carleman estimates, partial-boundary solvability, and geometric optics
constructions require only bounded potentials.
\end{remark}

Inverse coefficient problems for parabolic equations have been studied
extensively in Euclidean domains.  Early works concerning the recovery of
time-independent zeroth-order coefficients include
\cite{Isakov1991,AvdoninSeidman1995}.  The determination of coefficients
depending on both space and time introduces additional difficulties due to the
causal nature of the parabolic equation.  Choulli and Kian first established
stability for a time-dependent coefficient of the form \(\sigma(t)f(x)\) in
\cite{ChoulliKian2013}, and later obtained logarithmic stability for a general
time-dependent zeroth-order coefficient from a partial Dirichlet-to-Neumann map
in \cite{ChoulliKian2018}.  Uniqueness from partial boundary measurements was
subsequently studied by Fan and Duan in \cite{FanDuan2021}.  For earlier
coefficient recovery for heat equations from boundary measurements on two
boundary portions, see Canuto and Kavian \cite{CanutoKavian2001}.  More
recently, Feizmohammadi, Kian, and Uhlmann studied local partial-data problems
for reaction--diffusion and heat equations using spherical quasimodes and
proved, among other results, injectivity of the linearized partial
Dirichlet-to-Neumann map for the heat equation
\cite{FeizmohammadiKianUhlmann2024}.  Related inverse
problems for first-order perturbations and convection--diffusion equations have
also received considerable attention; see, for instance,
\cite{BellassouedRassas2020,SahooVashisth2020,BellassouedFraj2021,
FeizmohammadiKianUhlmann2022} and the references therein.
For a genuinely local-data problem in a Euclidean domain, Kumar and Purohit
recovered time-dependent convection and density coefficients, up to the natural
gauge, under a flatness assumption on the inaccessible part of the boundary
\cite{KumarPurohit2025}.  Their proof is based on a reflection argument.  This
differs from the present CTA setting, where no flat inaccessible boundary
portion is assumed and Gaussian beam quasimodes are used to handle the
transversal geometry.

The use of boundary Carleman estimates to isolate accessible and inaccessible
boundary portions has important precedents in elliptic partial-data inverse
problems, notably \cite{BukhgeimUhlmann2002,KenigSjostrandUhlmann2007}.
The corresponding geometric inverse problems on Riemannian manifolds are
comparatively less developed.  An important class of geometries in this
context is provided by manifolds admitting limiting Carleman weights.  Dos
Santos Ferreira, Kenig, Salo, and Uhlmann introduced the admissible geometric
framework and established uniqueness results for anisotropic inverse problems
by reducing the problem to an attenuated geodesic ray transform on a simple
transversal manifold \cite{DosSantosFerreiraKenigSaloUhlmann2009}.  Partial
data problems on such manifolds were further investigated by Kenig and Salo
\cite{KenigSalo2013}, while inverse problems for advection--diffusion operators
in admissible geometries were studied by Krupchyk and Uhlmann
\cite{KrupchykUhlmann2018}.  More recently, Mishra, Purohit, and Vashisth
considered a partial data inverse problem for a time-dependent
convection--diffusion equation on admissible manifolds and obtained recovery of
the time-dependent coefficients modulo the natural gauge invariance
\cite{MishraPurohitVashisth2025}.  In these admissible settings, the simplicity
of the transversal manifold plays an important role, both in the construction
of special solutions and in the injectivity of the associated geodesic ray
transforms.

A substantially more general geometric setting was introduced in the study of
the Calder\'on problem on conformally transversally anisotropic manifolds.  In
\cite{DosSantosFerreiraEtAl2016}, Dos Santos Ferreira, Kurylev, Lassas, and Salo
removed the simplicity assumption on the transversal manifold and constructed
complex geometrical optics solutions using Gaussian beam quasimodes
concentrating near non-tangential geodesics.  Their argument relates boundary
measurements to Fourier transforms in the Euclidean direction and geodesic ray
transforms, or more generally semiclassical defect measures, on the transversal
manifold.  In particular, uniqueness follows whenever the relevant geodesic ray
transform is injective.  Related ideas have recently been adapted to
time-dependent inverse problems for hyperbolic equations on CTA manifolds.
Liu, Saksala, and Yan recovered time-dependent damping and potential terms from
partial Cauchy data under an injectivity assumption for an attenuated geodesic
ray transform in \cite{LiuSaksalaYan2024}; see also
\cite{LiuSaksalaYan2025} for the recovery of a time-dependent potential on CTA
manifolds.

Very recently, Liu and Purohit studied the recovery of time-dependent
convection and density coefficients on CTA manifolds from partial measurements
\cite{LiuPurohit2026}.  Their input--output operator, however, allows
Dirichlet data prescribed on the full lateral boundary.  In the scalar
potential problem considered here, the Dirichlet inputs themselves are
restricted to the prescribed neighborhood \(U\).  Consequently, even after
setting the convection term equal to zero in their result, equality of the
restricted input--output maps considered here is insufficient to invoke their
full-input uniqueness theorem.  The present result therefore addresses the
scalar-potential case with locally supported Dirichlet inputs.  The additional
ingredient is the construction of two geometric optics solutions whose lateral
traces are supported in \(U\) and \(V\), respectively, so that the remaining
lateral boundary term in the integral identity vanishes exactly.  Accordingly,
the two uniqueness results are not direct special cases of one another: their
theorem treats simultaneous first- and zeroth-order coefficient recovery with
full lateral Dirichlet input, whereas ours treats scalar-potential recovery
under a genuinely local input condition.

The present work develops a local-input scalar-potential counterpart of this
Gaussian-beam approach for parabolic equations.  In
contrast with inverse parabolic results on admissible manifolds, we do not
impose simplicity on the transversal manifold \((M_0,g_0)\).  Instead, we
assume the injectivity of a family of attenuated geodesic ray transforms as in
\cref{ass:ray-injectivity}.  This change of geometry makes the usual
construction based on global polar normal coordinates unavailable and leads
naturally to the use of Gaussian beam quasimodes concentrated along
non-tangential geodesics.  Combining these quasimodes with boundary Carleman
estimates for the forward and adjoint parabolic operators, we construct
geometric optics solutions whose lateral traces are supported in the prescribed
neighborhoods \(U\) and \(V\).  The resulting integral identity then converts
equality of the partial input--output maps into the vanishing of the
corresponding attenuated geodesic ray transforms, which yields the desired
uniqueness.

The proof proceeds as follows.  In \cref{sec:preliminaries} we reduce both the
direct and inverse problems on the CTA metric to the product case.  In
\cref{sec:gaussian-beams} we construct transversal Gaussian beam quasimodes.
Boundary Carleman solvability then gives exact geometric optics solutions with
lateral traces supported in \(U\) and \(V\), respectively; see
\cref{sec:go-solutions}.  In \cref{sec:uniqueness}, the equality of the partial
maps produces an integral identity, and \cref{ass:ray-injectivity} completes the
argument.  The well-posedness and Carleman estimates are proved in the appendices.

\section{Conformal reduction to the product case}
\label{sec:preliminaries}

If \(H\) is a Hilbert space, \(H'\) denotes its continuous dual.  We write
\((F,u)_{H',H}\) for duality.  Unless explicitly stated otherwise, complex inner
products and duality pairings are conjugate-linear in the first argument and
linear in the second.  The Riemannian volume and boundary measures are denoted by
\(dV_g\) and \(dS_g\), and
\[
 \nabla_gv=g^{jk}\partial_{x_j}v\,\partial_{x_k},
 \qquad \partial_\nu v=\langle\nabla_gv,\nu\rangle_g.
\]
The analogous notation with subscript \(g_0\) is used on \(M_0\).

The purpose of this section is to reduce both the direct problem and the inverse
problem to the case \(c=1\).  The functional-analytic justification for
transposition solutions and generalized traces is deferred entirely to
Appendix~\ref{app:well-posedness}.

Set
\begin{equation}
 \widetilde g:=c^{-1}g=dx_1^2\oplus g_0,
 \qquad \rho:=\frac{n-2}{4},
 \label{eq:reduced-metric}
\end{equation}
and define
\begin{equation}
 \widetilde q
 :=c\bigl(q-c^\rho\Delta_g(c^{-\rho})\bigr)
 =cq+c^{-\rho}\Delta_{\widetilde g}(c^\rho).
 \label{eq:transformed-potential}
\end{equation}
The conformal relations are
\begin{equation}
 dV_g=c^{n/2}dV_{\widetilde g},\qquad
 dS_g=c^{(n-1)/2}dS_{\widetilde g},\qquad
 \nu_g=c^{-1/2}\nu_{\widetilde g}.
 \label{eq:conformal-measures}
\end{equation}
A direct calculation gives
\begin{equation}
 c^{(n+2)/4}\cL_{c,g,q}(c^{-\rho}v)
  =(\partial_t-\Delta_{\widetilde g}+\widetilde q)v,
 \qquad
 \cL_{c,g,q}:=c^{-1}\partial_t-\Delta_g+q.
 \label{eq:conformal-conjugation}
\end{equation}
Hence \(u\) corresponds to \(\widetilde u=c^\rho u\), with transformed data
\begin{equation}
 \widetilde u_0=c^\rho u_0,\qquad \widetilde f=c^\rho f.
 \label{eq:transformed-inputs}
\end{equation}
More precisely, \(u\) solves \eqref{eq:original-ibvp} if and only if
\(\widetilde u=c^\rho u\) solves
\begin{equation}
 \begin{cases}
 (\partial_t-\Delta_{\widetilde g}+\widetilde q)\widetilde u=0
     &\text{in }Q,\\
 \widetilde\tau_0\widetilde u=\widetilde f
     &\text{on }\Sigma,\\
 \widetilde r_0\widetilde u=\widetilde u_0
     &\text{in }M.
 \end{cases}
 \label{eq:reduced-ibvp}
\end{equation}
Thus the direct problem for the original CTA metric is equivalent to a direct
problem with product metric.
The signs of the front and back faces do not change, since
\(\partial_{\nu_g}\phi=c^{-1/2}\partial_{\nu_{\widetilde g}}\phi\).  Thus the
same sets \(U\) and \(V\) are used before and after the reduction.

The terminal and Neumann traces transform as
\begin{align}
 \widetilde r_T\widetilde u&=c^\rho r_Tu,
 \label{eq:transformed-terminal}\\
 \widetilde\tau_1\widetilde u
 &=c^{\rho+1/2}\tau_1u
   +\rho c^{-1}(\partial_{\nu_{\widetilde g}}c)\widetilde f.
 \label{eq:transformed-neumann}
\end{align}
The graph-space justification of these formulas, including transposition
solutions, is supplied in \cref{subsec:appendix-transfer}.

For brevity, write
\[
 \Lambda_{\widetilde g,\widetilde q}^{U,V}
 :=\Lambda_{1,\widetilde g,\widetilde q}^{U,V}.
\]
The input transformation \eqref{eq:transformed-inputs} is bijective, and the
second term in \eqref{eq:transformed-neumann} is known and independent of the
potential.  We therefore obtain the following equivalence.

\begin{proposition}\label{prop:map-reduction}
Let \(q_1,q_2\in C(\overline Q;\C)\), and let \(\widetilde q_1,\widetilde q_2\) be
defined by \eqref{eq:transformed-potential}.  Then
\begin{equation}
 \Lambda_{c,g,q_1}^{U,V}=\Lambda_{c,g,q_2}^{U,V}
 \quad\Longleftrightarrow\quad
 \Lambda_{\widetilde g,\widetilde q_1}^{U,V}
 =\Lambda_{\widetilde g,\widetilde q_2}^{U,V}.
 \label{eq:map-equivalence}
\end{equation}
Moreover,
\begin{equation}
 \widetilde q_1-\widetilde q_2=c(q_1-q_2).
 \label{eq:potential-difference-transform}
\end{equation}
\end{proposition}

\begin{proof}
The conjugation identity \eqref{eq:conformal-conjugation} and
\eqref{eq:transformed-inputs} give a bijective correspondence between the
inputs and solutions of the two direct problems.  The output relations
\cref{eq:transformed-terminal,eq:transformed-neumann} then show that equality
of the two original partial input--output maps is equivalent to equality of the
two reduced maps; the additional term in \eqref{eq:transformed-neumann} depends
only on the known conformal factor and the prescribed Dirichlet datum.  Finally,
\eqref{eq:transformed-potential} gives
\(\widetilde q_1-\widetilde q_2=c(q_1-q_2)\).  The validity of these
transformations for generalized traces and transposition solutions is proved in
Appendix~\ref{app:well-posedness}.
\end{proof}

In view of \cref{prop:map-reduction}, from now on we assume
\begin{equation}
 g=dx_1^2\oplus g_0
 \label{eq:product-metric-standing}
\end{equation}
and omit all tildes.  We use throughout
\begin{equation}
 P_q:=\partial_t-\Delta_g+q,\qquad
 P_q^*:=-\partial_t-\Delta_g+\overline q,\qquad
 P_q^{\mathrm t}:=-\partial_t-\Delta_g+q=P_{\overline q}^*.
 \label{eq:operators-adjoint-transpose}
\end{equation}
The last operator is the transpose for the complex-bilinear Green identity.
Finally, define the graph spaces needed in the remainder of the paper by
\begin{equation}
 \cH_\pm(Q):=
 \{u\in L^2(Q):(\pm\partial_t-\Delta_g)u\in L^2(Q)\},
 \qquad
 \|u\|_{\cH_\pm}^2
 :=\|u\|_{L^2(Q)}^2
 +\|(\pm\partial_t-\Delta_g)u\|_{L^2(Q)}^2.
 \label{eq:reduced-graph-spaces}
\end{equation}
The generalized lateral and temporal trace maps on these graph spaces are
constructed in Appendix~\ref{app:well-posedness}.

\section{Parabolic Gaussian beam quasimodes}
\label{sec:gaussian-beams}

We now construct Gaussian beam quasimodes adapted to the forward parabolic
operator and its adjoint.  The concentration occurs in the transversal variables,
but we state the result for the conjugated parabolic operators used later.  The
construction follows the Gaussian beam arguments in
\cite{Ralston1982,DosSantosFerreiraEtAl2016,LiuSaksalaYan2025}.

Throughout Sections~\ref{sec:gaussian-beams}--\ref{sec:uniqueness} and
Appendix~\ref{app:carleman}, we work with a semiclassical parameter
\(h\in(0,h_0]\), where \(h_0>0\) is sufficiently small, and all asymptotic
estimates are understood as \(h\to0\).  We write
the conjugated operators in semiclassically scaled form; in particular, spatial
first derivatives are measured by \(h\nabla_g\), while all \(L^2\)-norms are
the standard unscaled norms.  For the general semiclassical notation and
conventions, we refer to \cite{Zworski2012}.

Fix
\begin{equation}
 \beta\in\left(\frac1{\sqrt3},1\right),\qquad
 \alpha:=\sqrt{1-\beta^2},
 \label{eq:beta-alpha}
\end{equation}
and distinguish the \(h\)-dependent lower-case weight from the spatial weight
\(\phi\) by writing
\begin{equation}
 \varphi_h(t,x):=\phi(x)+\frac{\beta^2t}{h}
 =x_1+\frac{\beta^2t}{h},\qquad
 \Phi_h:=\frac{\varphi_h}{h}
 =\frac{x_1}{h}+\frac{\beta^2t}{h^2}.
 \label{eq:parabolic-weights}
\end{equation}
In particular,
\begin{equation}
 \partial_\nu\varphi_h=\partial_\nu\phi=\partial_\nu x_1.
 \label{eq:weight-normal-derivative}
\end{equation}
The semiclassically scaled conjugated operators are
\begin{align}
 \cL_{+,h}
 &:=h^2e^{-\Phi_h}P_qe^{\Phi_h}
 =h^2\partial_t-h^2\Delta_g-2h\partial_{x_1}
   -\alpha^2+h^2q,
 \label{eq:L-plus}\\
 \cL_{-,h}
 &:=h^2e^{\Phi_h}P_q^*e^{-\Phi_h}
 =-h^2\partial_t-h^2\Delta_g+2h\partial_{x_1}
   -\alpha^2+h^2\overline q.
 \label{eq:L-minus}
\end{align}

\subsection{Construction and concentration}

Let \(\gamma:[0,L]\to M_0\) be a unit-speed non-tangential geodesic.  Set
\begin{equation}
 s=\frac1h+i\sigma,\qquad \eta=\alpha^2\sigma,\qquad \sigma\in\R.
 \label{eq:s-eta}
\end{equation}
Whenever the signs \(\pm\) occur in one statement, the upper and lower signs are
taken consistently.

\begin{proposition}[Parabolic Gaussian beam quasimodes]
\label{prop:parabolic-quasimodes}
Let \(q\in L^\infty(Q)\), let \(\gamma:[0,L]\to M_0\) be a unit-speed
non-tangential geodesic, fix \(\sigma\in\R\), and prescribe
\(\zeta_\pm\in C^\infty([0,T])\).  There exist families
\[
 V_{\pm,h}\in C^\infty([0,T]\times\R\times M_0)
\]
whose restrictions to \(\overline Q\) satisfy, as \(h\to0\),
\begin{equation}
 \|V_{\pm,h}\|_{L^2(Q)}=O(1),\qquad
 \|\cL_{\pm,h}V_{\pm,h}\|_{L^2(Q)}=O(h^{3/2}).
 \label{eq:quasimode-estimates}
\end{equation}
The estimates are uniform for \(\sigma\) in a fixed compact subset of \(\R\).
\end{proposition}

\begin{proposition}[Concentration property]
\label{prop:quasimode-concentration}
Under the assumptions of \cref{prop:parabolic-quasimodes}, the families
\(V_{\pm,h}\) may be normalized so that, for every
\(F\in C([0,T]\times\R\times M_0)\) compactly supported in \(x_1\),
\begin{align}
 &\lim_{h\to0}\int_0^T\int_\R\int_{M_0}
 F(t,x_1,x')V_{+,h}(t,x_1,x')V_{-,h}(t,x_1,x')
 \,dV_{g_0}(x')\,dx_1\,dt \notag\\
 &\quad=\int_0^T\int_\R\int_0^L
 F(t,x_1,\gamma(r))\zeta_+(t)\zeta_-(t)
 e^{2i\alpha^2\sigma x_1}e^{-2\alpha\sigma r}
 \,dr\,dx_1\,dt.
 \label{eq:quasimode-concentration}
\end{align}
\end{proposition}

\begin{remark}
The constant attenuation in \eqref{eq:quasimode-concentration} is
\(\alpha\sigma\) at the level of one beam and \(2\alpha\sigma\) for the product.
Since \(\alpha>0\), varying \(\sigma\) near zero produces all sufficiently small
attenuations required in \cref{ass:ray-injectivity}.
\end{remark}

Introduce the transversal semiclassical operator
\begin{equation}
 P_{\perp,s}:=-h^2\Delta_{g_0}-\alpha^2h^2s^2.
 \label{eq:transversal-operator}
\end{equation}
The two ingredients needed for the parabolic construction are isolated in the
following lemmas.

\begin{lemma}[Global transversal quasimode]
\label{lem:global-transversal-quasimode}
Let \(\gamma:[0,L]\to M_0\) be a unit-speed non-tangential geodesic and let
\(s=h^{-1}+i\sigma\).  There exists \(v_s\in C^\infty(M_0)\) such that
\begin{equation}
 \|v_s\|_{L^2(M_0)}+\|h\nabla_{g_0}v_s\|_{L^2(M_0)}=O(1),
 \qquad
 \|P_{\perp,s}v_s\|_{L^2(M_0)}=O(h^{3/2}).
 \label{eq:transversal-quasimode-estimates}
\end{equation}
The estimates are uniform for \(\sigma\) in a fixed compact subset of \(\R\).
\end{lemma}

\begin{lemma}[Concentration of the transversal beam]
\label{lem:transversal-concentration}
The quasimodes in \cref{lem:global-transversal-quasimode} can be normalized so
that, for every \(\psi\in C(M_0)\),
\begin{equation}
 \lim_{h\to0}\int_{M_0}|v_s(x')|^2\psi(x')\,dV_{g_0}(x')
 =\int_0^L e^{-2\alpha\sigma r}\psi(\gamma(r))\,dr.
 \label{eq:transversal-concentration}
\end{equation}
\end{lemma}

\begin{proof}[Proof of \cref{lem:global-transversal-quasimode}]
The construction is similar to the Gaussian beam constructions in
\cite{DosSantosFerreiraEtAl2016,LiuSaksalaYan2025}.  We include the details
needed for the residual estimate and for globalization across
self-intersections.

In this proof, unlabelled \(L^2\)-norms are taken with respect to
\(dV_{g_0}\) over the relevant local or tubular neighborhood; global norms
are indicated by \(L^2(M_0)\).

\noindent\emph{Step 1: Local Gaussian beam.}
Write \(m=\dim M_0=n-1\) and \(d=m-1=n-2\).  Extend \((M_0,g_0)\) to a
closed manifold \((\widehat M_0,g_0)\), and extend \(\gamma\) as a unit-speed
geodesic.  Since \(\gamma\) is non-tangential, there is \(\epsilon>0\) such that
\begin{equation}
 \gamma(r)\in\widehat M_0\setminus M_0,\qquad
 r\in[-2\epsilon,0)\cup(L,L+2\epsilon].
 \label{eq:extended-geodesic-outside}
\end{equation}
We may arrange that the segment
\(\gamma:[-2\epsilon,L+2\epsilon]\to\widehat M_0\) is not a loop.

Following the Fermi-coordinate construction along a non-tangential geodesic
in \cite[Section~7]{KenigSalo2013}, fix
\(r_*\in[-\epsilon,L+\epsilon]\) and choose, in a neighborhood of
\(\gamma(r_*)\), Fermi coordinates
\[
 (r,y)\in I\times B(0,\delta),\qquad y\in\R^d,
\]
such that
\begin{equation}
 \gamma(r)=(r,0),\qquad
 (g_0)_{jk}(r,0)=\delta_{jk},\qquad
 \partial_\ell(g_0)_{jk}(r,0)=0.
 \label{eq:fermi-properties}
\end{equation}
The quadratic part of the inverse metric coefficient is
\begin{equation}
 g_0^{rr}(r,y)=1-\mathsf F(r)y\cdot y+O(|y|^3),
 \label{eq:inverse-metric-expansion}
\end{equation}
where \(\mathsf F(r)\) is a smooth real symmetric \(d\times d\) matrix.  Let
\(H(r)\) solve
\begin{equation}
 \dot H+H^2=\mathsf F,\qquad H(r_*)=H_*,
 \label{eq:riccati}
\end{equation}
where \(H_*\) is complex symmetric with positive-definite imaginary part.
The standard matrix Riccati result
\cite[Lemma~2.56]{KatchalovKurylevLassas2001} gives a unique smooth complex
symmetric solution along the full extended segment such that
\(\operatorname{Im}H(r)>0\).
Define
\begin{equation}
 \Theta(r,y):=\alpha\left(r+\frac12H(r)y\cdot y\right).
 \label{eq:gaussian-phase}
\end{equation}
Then
\begin{align}
 \langle d\Theta,d\Theta\rangle_{g_0}-\alpha^2&=O(|y|^3),
 \label{eq:eikonal-error}\\
 \Theta(r,0)=\alpha r,\qquad
 d\Theta(r,0)&=\alpha\dot\gamma(r)^\flat.
 \label{eq:phase-on-ray}
\end{align}
Here \(\dot\gamma(r)^\flat:=g_0(\dot\gamma(r),\,\cdot\,)\) denotes the metric
dual of \(\dot\gamma(r)\), and \(d\Theta(r,0)\) is the differential of
\(\Theta\) at the point \(\gamma(r)=(r,0)\).  Since the Fermi coordinates
are adapted to the unit-speed geodesic \(\gamma\), the second identity in
\eqref{eq:phase-on-ray} is equivalently
\[
 \partial_r\Theta(r,0)=\alpha,\qquad
 \partial_{y^j}\Theta(r,0)=0,\quad j=1,\ldots,d.
\]
After reducing \(\delta\), there is \(c_0>0\) such that
\begin{equation}
 \operatorname{Im}\Theta(r,y)\geq c_0|y|^2,\qquad |y|<\delta.
 \label{eq:phase-positivity}
\end{equation}

Choose \(\chi\in C_0^\infty(\R^d)\) equal to one for \(|y|\leq1/4\) and zero
for \(|y|\geq1/2\), and set
\begin{equation}
 v_s^{\mathrm{loc}}(r,y)
 :=e^{is\Theta(r,y)}a_h(r,y),\qquad
 a_h(r,y):=h^{-d/4}a_0(r)\chi(y/\delta).
 \label{eq:local-beam}
\end{equation}
The leading amplitude is determined by
\begin{equation}
 \dot a_0(r)+\frac12\operatorname{tr}H(r)a_0(r)=0.
 \label{eq:transport}
\end{equation}
Near \(y=0\),
\begin{equation}
 2\langle d\Theta,da_h\rangle_{g_0}
 +(\Delta_{g_0}\Theta)a_h=h^{-d/4}O(|y|).
 \label{eq:transport-error}
\end{equation}
Terms in which the cutoff is differentiated are \(O(h^N)\) in \(L^2\) for
every \(N>0\), by \eqref{eq:phase-positivity}.  By the change of variables
\(y=h^{1/2}z\) in the Gaussian integral, we have
\begin{equation}
 \bigl\|h^{-d/4}|y|^ke^{-c|y|^2/h}\bigr\|_{L^2(\R^d)}
 =O(h^{k/2}),\qquad k\geq0,
 \label{eq:gaussian-scaling}
\end{equation}
and therefore
\begin{equation}
 \|v_s^{\mathrm{loc}}\|_{L^2}
 +\|h\nabla_{g_0}v_s^{\mathrm{loc}}\|_{L^2}=O(1).
 \label{eq:local-beam-bound}
\end{equation}
A direct calculation gives
\begin{align}
 P_{\perp,s}(e^{is\Theta}a_h)
 =e^{is\Theta}\bigl[
 &h^2s^2(\langle d\Theta,d\Theta\rangle_{g_0}-\alpha^2)a_h\notag\\
 &-ih^2s\{2\langle d\Theta,da_h\rangle_{g_0}
       +(\Delta_{g_0}\Theta)a_h\}
 -h^2\Delta_{g_0}a_h\bigr].
 \label{eq:local-residual-expansion}
\end{align}
We now estimate the three terms in \eqref{eq:local-residual-expansion}.
Since \(s=h^{-1}+i\sigma\), uniformly for \(\sigma\) in a fixed compact
subset of \(\R\),
\[
 |h^2s^2|\leq C,\qquad |h^2s|\leq Ch.
\]
Moreover,
\[
 |e^{is\Theta(r,y)}|
 =\exp\left(-\frac{\operatorname{Im}\Theta(r,y)}{h}
            -\sigma\operatorname{Re}\Theta(r,y)\right)
 \leq Ce^{-c|y|^2/h},
\]
where we used \eqref{eq:phase-positivity} and the boundedness of
\(\operatorname{Re}\Theta\) on the support of the local beam.  It follows from
\cref{eq:eikonal-error,eq:gaussian-scaling} that the eikonal error satisfies
\begin{align*}
 &\bigl\|e^{is\Theta}h^2s^2
 \bigl(\langle d\Theta,d\Theta\rangle_{g_0}-\alpha^2\bigr)a_h
 \bigr\|_{L^2}\\
 &\qquad\leq
 C\bigl\|h^{-d/4}|y|^3e^{-c|y|^2/h}\bigr\|_{L^2}
 =O(h^{3/2}).
\end{align*}
Similarly, \cref{eq:transport-error,eq:gaussian-scaling} give
\begin{align*}
 &\bigl\|e^{is\Theta}h^2s
 \{2\langle d\Theta,da_h\rangle_{g_0}
       +(\Delta_{g_0}\Theta)a_h\}\bigr\|_{L^2}\\
 &\qquad\leq
 Ch\bigl\|h^{-d/4}|y|e^{-c|y|^2/h}\bigr\|_{L^2}
 =O(h^{3/2}).
\end{align*}
The \(r\)-interval is fixed, so its integration only contributes a uniform
constant to these estimates.  Finally, since the cutoff is independent of
\(h\), the smoothness of \(a_0\) and the metric coefficients gives
\[
 h^2\bigl\|e^{is\Theta}\Delta_{g_0}a_h\bigr\|_{L^2}
 \leq Ch^2
 \bigl\|h^{-d/4}e^{-c|y|^2/h}\bigr\|_{L^2}
 =O(h^2).
\]
As noted above, every term in which a derivative falls on the cutoff is
\(O(h^N)\) for all \(N>0\).  Combining the three bounds yields
\begin{equation}
 \|P_{\perp,s}v_s^{\mathrm{loc}}\|_{L^2}=O(h^{3/2}).
 \label{eq:local-residual}
\end{equation}

\smallskip
\noindent\emph{Step 2: Globalization along the geodesic.}
After decreasing \(\epsilon\) if necessary, the endpoints \(-\epsilon\) and
\(L+\epsilon\) are not self-intersection times.  Since \(\widehat M_0\) is
compact and the extended geodesic segment is not a loop,
\cite[Lemma~7.2]{KenigSalo2013} shows that it has only finitely many
self-intersection times in
\([-\epsilon,L+\epsilon]\).  If such times occur, write them as
\[
 -\epsilon=r_0<r_1<\cdots<r_N<r_{N+1}=L+\epsilon.
\]
By \cite[Lemma~7.4]{KenigSalo2013}, there is a Fermi-coordinate cover
\[
 \{(\Omega_j,\Psi_j)\}_{j=0}^{N+1}
 \quad\text{of}\quad \gamma([-\epsilon,L+\epsilon])
\]
having the following properties:
\begin{enumerate}[label=\textup{(\roman*)}]
\item \(\Psi_j(\Omega_j)=I_j\times B(0,\delta)\), where each \(I_j\) is an
open interval and \(B(0,\delta)\subset\R^d\); the same \(\delta>0\) may be
used for every \(j\) and may be chosen arbitrarily small;
\item \(\Psi_j(\gamma(r))=(r,0)\) for \(r\in I_j\);
\item each \(r_j\), \(1\leq j\leq N\), belongs only to \(I_j\), and
\(I_j\cap I_k=\varnothing\) whenever \(|j-k|>1\);
\item whenever \(I_j\cap I_k\neq\varnothing\), the two coordinate maps
agree on the corresponding overlap,
\[
 \Psi_j=\Psi_k
 \quad\text{on}\quad
 \Psi_j^{-1}\bigl((I_j\cap I_k)\times B(0,\delta)\bigr).
\]
\end{enumerate}
More explicitly, for some sufficiently small \(\delta_1>0\), the intervals
may be taken as
\[
 \begin{split}
 I_0&=(-2\epsilon,r_1-\delta_1),\\
 I_j&=(r_j-2\delta_1,r_{j+1}-\delta_1),
       \qquad 1\leq j\leq N,\\
 I_{N+1}&=(r_{N+1}-2\delta_1,L+2\epsilon).
 \end{split}
\]
In this case set \(J=N+1\).  If the extended segment has no
self-intersections, a single tubular Fermi chart satisfying
\textup{(i)}--\textup{(ii)} is sufficient; we then take
\(I_0=(-2\epsilon,L+2\epsilon)\) and set \(J=0\).

Starting in \(\Omega_0\), construct \(v_s^{(0)}\) as in Step~1.  For
\(j<J\), continue the phase and amplitude into \(\Omega_{j+1}\), using their
values at a point of \(I_j\cap I_{j+1}\) as Cauchy data.

Uniqueness for
the Riccati and transport equations, together with property
\textup{(iv)}, implies that \(v_s^{(j)}=v_s^{(j+1)}\) on the corresponding
overlap.  Indeed, the phases, leading amplitudes, and transverse cutoffs are
written in the same Fermi coordinates there and have the same initial data.
Continuing in this way gives compatible local beams along each parameterized
passage of the geodesic.

Let \(\{\vartheta_j\}_{j=0}^{J}\) be a smooth partition of unity on
\([-\epsilon,L+\epsilon]\) subordinate to \(\{I_j\}_{j=0}^{J}\).  We use
the same notation for its extension to \(\Omega_j\) through the projection
\((r,y)\mapsto r\).  Extending each summand by zero outside \(\Omega_j\), set
\begin{equation}
 v_s=\sum_{j=0}^{J}\vartheta_jv_s^{(j)}.
 \label{eq:global-beam-embedded}
\end{equation}
Because the interval cover has multiplicity at most two, only two consecutive
cutoffs can be active on a parameter overlap.  The equality of the two local
beams and the identity \(\vartheta_j+\vartheta_{j+1}=1\) show that the sum in
\eqref{eq:global-beam-embedded} reduces there to the single beam associated
with that passage.  Thus the partition of unity introduces no new error term.

It remains only to describe what happens where distinct passages meet.  Let
\(z_1,\ldots,z_R\) be the distinct self-intersection points of \(\gamma\), and
write
\[
 \gamma^{-1}(\{z_k\})\cap[-\epsilon,L+\epsilon]
 =\{r_{k,1},\ldots,r_{k,N_k}\}.
\]
Thus \(N_k\) is the number of passages of \(\gamma\) through \(z_k\).  Choose
pairwise disjoint neighborhoods \(V_k\) of the points \(z_k\).  After
shrinking them, \(\gamma\) is embedded on a small parameter interval about
each \(r_{k,\ell}\).  We denote the image of this interval in \(V_k\) by
\(\Gamma_{k,\ell}\), \(1\leq\ell\leq N_k\); it is the branch segment
corresponding to the \(\ell\)-th passage through \(z_k\).
Let \(v_s^{(k,\ell)}\) denote the local beam in
\eqref{eq:global-beam-embedded} associated with this passage.  Local beams
belonging to consecutive charts along the same branch have already been
identified by the preceding argument, whereas beams belonging to distinct
branches are kept as separate summands.  The compact part of the remaining
trace can be covered by finitely many neighborhoods
\(W_1,\ldots,W_S\), where \(W_\mu\subset\Omega_{j(\mu)}\) for some
\(j(\mu)\).  Reducing \(\delta\) once more if necessary, we may arrange that
\[
 \operatorname{supp}(v_s)\cap M_0
 \subset \bigcup_{k=1}^{R}V_k\cup\bigcup_{\mu=1}^{S}W_\mu.
\]
On this cover,
\begin{equation}
 v_s|_{V_k}=\sum_{\ell=1}^{N_k}v_s^{(k,\ell)},
 \qquad
 v_s|_{W_\mu}=v_s^{(j(\mu))}.
 \label{eq:beam-at-self-intersection}
\end{equation}
On each \(W_\mu\), the estimates follow directly from
\cref{eq:local-beam-bound,eq:local-residual}.  On each \(V_k\), the triangle
inequality and \eqref{eq:beam-at-self-intersection} give
\begin{align*}
 \|v_s\|_{L^2(V_k)}+\|h\nabla_{g_0}v_s\|_{L^2(V_k)}
 &\leq \sum_{\ell=1}^{N_k}
 \bigl(\|v_s^{(k,\ell)}\|_{L^2(V_k)}
 +\|h\nabla_{g_0}v_s^{(k,\ell)}\|_{L^2(V_k)}\bigr)\leq C,\\
 \|P_{\perp,s}v_s\|_{L^2(V_k)}
 &\leq\sum_{\ell=1}^{N_k}
 \|P_{\perp,s}v_s^{(k,\ell)}\|_{L^2(V_k)}\leq Ch^{3/2}.
\end{align*}
Since the cover and all branch families are finite, these local estimates give
\begin{align}
 \|v_s\|_{L^2(M_0)}+\|h\nabla_{g_0}v_s\|_{L^2(M_0)}&\leq C,
 \label{eq:global-beam-L2-bound}\\
 \|P_{\perp,s}v_s\|_{L^2(M_0)}&\leq Ch^{3/2}.
 \label{eq:global-beam-residual-bound}
\end{align}
No orthogonality between distinct branches is needed for these two estimates.
This proves \cref{lem:global-transversal-quasimode}.
\end{proof}

\begin{proof}[Proof of \cref{lem:transversal-concentration}]
We retain the phase and amplitude notation from the preceding proof.  Unless
a domain is displayed, the \(L^2\)-norms below are taken over the local
neighborhood currently under consideration.

Write \(H=A+iB\), where \(A=\operatorname{Re}H\) and
\(B=\operatorname{Im}H\), and a dot denotes differentiation with respect to
\(r\).  We first normalize the beam on a single branch.
Taking imaginary parts in \eqref{eq:riccati} gives
\[
 \dot B+AB+BA=0,
\]
and Jacobi's formula, or equivalently
\cite[Lemma~2.58]{KatchalovKurylevLassas2001}, gives
\begin{equation}
 \det B(r)=\det B(r_*)
 \exp\left(-2\int_{r_*}^r\operatorname{tr}A(\tau)\,d\tau\right).
 \label{eq:detB}
\end{equation}
Equation \eqref{eq:transport} gives
\begin{equation}
 |a_0(r)|^2=|a_0(r_*)|^2
 \exp\left(-\int_{r_*}^r\operatorname{tr}A(\tau)\,d\tau\right).
 \label{eq:amplitude-modulus}
\end{equation}
Choose
\begin{equation}
 a_0(r_*)=\left(\frac{\alpha}{\pi}\right)^{d/4}
          \bigl(\det B(r_*)\bigr)^{1/4}.
 \label{eq:amplitude-normalization}
\end{equation}
Then
\begin{equation}
 \left(\frac{\pi}{\alpha}\right)^{d/2}
 \frac{|a_0(r)|^2}{\sqrt{\det B(r)}}=1.
 \label{eq:normalization-identity}
\end{equation}

Recall from \eqref{eq:beam-at-self-intersection} and the support inclusion
preceding it that \(\operatorname{supp}(v_s)\cap M_0\) is covered by the
finitely many sets \(V_k\) and \(W_\mu\).  Choose a smooth partition of unity
\[
 \{\rho_k^V\}_{k=1}^R\cup\{\rho_\mu^W\}_{\mu=1}^S
\]
subordinate to this cover, with their sum equal to one on a neighborhood of
\(\operatorname{supp}(v_s)\cap M_0\).  By linearity, it is enough to prove
the asserted limit with \(\psi\) replaced by each of
\(\rho_k^V\psi\) and \(\rho_\mu^W\psi\).  Thus, below, \(\psi\) denotes a
localized test function compactly supported relative to one set
\(W_\mu\cap M_0\) or one set \(V_k\cap M_0\).  Such a function may still be
nonzero on the physical boundary \(\partial M_0\).

First suppose that \(\psi\) is supported in \(W_\mu\cap M_0\).  By
\eqref{eq:beam-at-self-intersection},
\(v_s=v_s^{(j(\mu))}\) on the support of \(\psi\).  In the corresponding
Fermi chart, write \(dV_{g_0}=\mathcal J(r,y)\,dr\,dy\), where
\(\mathcal J(r,0)=1\).  By
\cref{eq:gaussian-phase,eq:local-beam},
\begin{align}
 |v_s^{(j(\mu))}(r,y)|^2
 ={}&h^{-d/2}|a_0(r)|^2e^{-2\alpha\sigma r}
 \exp\left(-\frac{\alpha}{h}B(r)y\cdot y\right)\notag\\
 &\times\exp\bigl(-\alpha\sigma A(r)y\cdot y\bigr)
 \chi(y/\delta)^2.
 \label{eq:beam-modulus}
\end{align}
After the change of variables \(y=h^{1/2}z\), dominated convergence and
\eqref{eq:normalization-identity} show that
\begin{align}
 \lim_{h\to0}\int_{M_0}|v_s(x')|^2\psi(x')\,dV_{g_0}(x')
 &=\int_{I_{j(\mu)}\cap[0,L]}
 e^{-2\alpha\sigma r}\psi(\gamma(r))\,dr\notag\\
 &=\int_0^L e^{-2\alpha\sigma r}\psi(\gamma(r))\,dr.
 \label{eq:local-concentration}
\end{align}
The last equality uses the fact that \(W_\mu\) meets only the branch
corresponding to \(I_{j(\mu)}\).

Now suppose that \(\psi\) is supported in \(V_k\cap M_0\).  In this case,
\eqref{eq:beam-at-self-intersection} gives, on the support of \(\psi\),
\[
 |v_s|^2
 =\sum_{\ell=1}^{N_k}|v_s^{(k,\ell)}|^2
  +\sum_{\ell\neq\ell'}
    v_s^{(k,\ell)}\overline{v_s^{(k,\ell')}}.
\]
The calculation leading to \eqref{eq:local-concentration}, applied to each
branch separately, gives
\[
 \lim_{h\to0}\sum_{\ell=1}^{N_k}
 \int_{M_0}|v_s^{(k,\ell)}|^2\psi\,dV_{g_0}
 =\int_0^L e^{-2\alpha\sigma r}\psi(\gamma(r))\,dr.
\]
Here the right-hand side counts all passages of \(\gamma\) through \(V_k\).
It remains to show that the cross terms tend to zero.

Fix \(\ell\neq\ell'\).  Suppressing \(k\) in the branchwise phase and
amplitude notation, for each \(\jmath\in\{\ell,\ell'\}\) write
\[
 r_\jmath:=r_{k,\jmath},\qquad
 v_s^{(k,\jmath)}=e^{is\Theta_\jmath}a_{\jmath,h}.
\]
Set
\begin{equation}
 \phi_{\ell\ell'}
 :=\operatorname{Re}\Theta_\ell-\operatorname{Re}\Theta_{\ell'}.
 \label{eq:cross-phase}
\end{equation}
The two branches have distinct tangent vectors at \(z_k\); otherwise uniqueness
for the geodesic equation would identify them locally.  Hence
\eqref{eq:phase-on-ray} gives
\begin{equation}
 d\phi_{\ell\ell'}(z_k)
 =\alpha\bigl(\dot\gamma(r_\ell)^\flat
 -\dot\gamma(r_{\ell'})^\flat\bigr)\neq0.
 \label{eq:cross-phase-gradient}
\end{equation}
Since there are only finitely many pairs of branches, the neighborhoods
\(V_k\) in the preceding proof may be chosen so small that, for some \(c>0\),
\[
 |d\phi_{\ell\ell'}|_{g_0}\geq c
 \quad\text{in }V_k
\]
for every \(\ell\neq\ell'\).  For
\(\jmath\in\{\ell,\ell'\}\), write
\[
 v_s^{(k,\jmath)}
 =e^{i\operatorname{Re}\Theta_\jmath/h}w_{\jmath,h},\qquad
 w_{\jmath,h}
 :=e^{-\operatorname{Im}\Theta_\jmath/h}
   e^{-\sigma\Theta_\jmath}a_{\jmath,h}.
\]
The cross term has the form
\begin{equation}
 \mathcal C_{\ell\ell'}(h;\psi)
 :=\int_{V_k\cap M_0}e^{i\phi_{\ell\ell'}/h}
 w_{\ell,h}\overline{w_{\ell',h}}\psi\,dV_{g_0}.
 \label{eq:cross-term-integral}
\end{equation}
Because \(\psi\) was obtained from the preceding partition of unity, it
vanishes near the artificial boundary
\(\partial V_k\cap M_0^{\mathrm{int}}\).

We first record the estimates needed below.  The local beam bound gives
\(\|w_{\ell,h}\|_{L^2(V_k\cap M_0)}=O(1)\).  In Fermi coordinates for the
\(\ell\)-th branch, \(|d\operatorname{Im}\Theta_\ell|\leq C|y|\), and hence
\begin{equation}
 \bigl\||d\operatorname{Im}\Theta_\ell|w_{\ell,h}\bigr\|_{L^2(V_k\cap M_0)}^2
 \leq Ch^{-d/2}\int_{\R^d}|y|^2e^{-c|y|^2/h}\,dy
 =O(h).
 \label{eq:weighted-beam-derivative}
\end{equation}
If \(V_k\) meets \(\partial M_0\), non-tangentiality allows the boundary to be
written in these coordinates as \(r=r(y)\).  The same Gaussian scaling gives
\begin{equation}
 \|w_{\ell,h}\|_{L^2(V_k\cap\partial M_0)}^2
 \leq Ch^{-d/2}\int_{\R^d}e^{-c|y|^2/h}\,dy
 =O(1).
 \label{eq:beam-boundary-bound}
\end{equation}
The corresponding estimates hold for the \(\ell'\)-th branch.  All constants
are uniform for \(\sigma\) in a fixed compact subset of \(\R\).

Let \(\varepsilon>0\).  Choose a smooth function \(\psi_1\), supported in
\(V_k\) and smooth up to \(\partial M_0\), such that
\(\psi=\psi_1+\psi_2\) and
\(\|\psi_2\|_{L^\infty(V_k\cap M_0)}\leq\varepsilon\).  By Cauchy--Schwarz,
\begin{equation}
 |\mathcal C_{\ell\ell'}(h;\psi_2)|
 \leq
 \|w_{\ell,h}\|_{L^2}
 \|w_{\ell',h}\|_{L^2}
 \|\psi_2\|_{L^\infty}
 \leq C\varepsilon.
 \label{eq:cross-term-continuous-part}
\end{equation}
For the smooth part, introduce the nonstationary-phase vector field
\begin{equation}
 X_{\ell\ell'}
 :=\frac{\nabla_{g_0}\phi_{\ell\ell'}}
         {|d\phi_{\ell\ell'}|_{g_0}^2},
 \qquad
 X_{\ell\ell'}e^{i\phi_{\ell\ell'}/h}
 =\frac{i}{h}e^{i\phi_{\ell\ell'}/h}.
 \label{eq:nonstationary-operator}
\end{equation}
One integration by parts gives
\begin{align}
 \mathcal C_{\ell\ell'}(h;\psi_1)
 ={}&\frac{h}{i}\int_{V_k\cap\partial M_0}
 e^{i\phi_{\ell\ell'}/h}
 \langle X_{\ell\ell'},\nu_0\rangle_{g_0}
 w_{\ell,h}\overline{w_{\ell',h}}\psi_1\,dS_{g_0}
 \notag\\
 &-\frac{h}{i}\int_{V_k\cap M_0}
 e^{i\phi_{\ell\ell'}/h}
 \operatorname{div}_{g_0}\bigl(
 X_{\ell\ell'}w_{\ell,h}\overline{w_{\ell',h}}\psi_1
 \bigr)\,dV_{g_0}.
 \label{eq:cross-term-integration-by-parts}
\end{align}
Here \(\nu_0\) is the outward unit normal to \(\partial M_0\).  The boundary
integral is absent if \(V_k\cap\partial M_0=\varnothing\), and otherwise it is
\(O(h)\) by \eqref{eq:beam-boundary-bound}.  In the interior integral, the
largest terms occur when a derivative falls on
\(e^{-\operatorname{Im}\Theta_\ell/h}\) or
\(e^{-\operatorname{Im}\Theta_{\ell'}/h}\).  Using
\eqref{eq:weighted-beam-derivative} and Cauchy--Schwarz, these terms are
\(O(h^{1/2})\).  Derivatives of the smooth coefficients, \(\psi_1\),
\(e^{-\sigma\Theta_\ell}\), and the leading amplitudes contribute \(O(h)\);
cutoff derivatives are exponentially small.  Therefore
\[
 |\mathcal C_{\ell\ell'}(h;\psi_1)|=O(h^{1/2}).
\]
Taking \(h\to0\) and then \(\varepsilon\to0\) in
\eqref{eq:cross-term-continuous-part} yields
\begin{equation}
 \mathcal C_{\ell\ell'}(h;\psi)=o(1).
 \label{eq:cross-term-decay}
\end{equation}
Thus every cross term vanishes, proving the desired limit for \(\psi\)
supported in \(V_k\).  Summing the local limits by the partition of unity
proves \eqref{eq:transversal-concentration}.  This proves
\cref{lem:transversal-concentration}.
\end{proof}

\begin{proof}[Proof of \cref{prop:parabolic-quasimodes}]
Take the transversal beam from \cref{lem:global-transversal-quasimode}.  The
transversal construction is independent of \(t\).  We may therefore let
\(\zeta_\pm\in C^\infty([0,T])\) be arbitrary fixed smooth functions and
attach them as time amplitudes only at this stage.  Set
\begin{equation}
 V_{+,h}(t,x)=\zeta_+(t)e^{i\eta x_1}v_s(x'),\qquad
 V_{-,h}(t,x)=\zeta_-(t)e^{i\eta x_1}\overline{v_s(x')}.
 \label{eq:parabolic-quasimode-ansatz}
\end{equation}
The subscript ``\(-\)'' refers to the conjugated adjoint operator
\(\cL_{-,h}\), not to the sign of the \(x_1\)-oscillation.  We keep the same
factor \(e^{i\eta x_1}\) and conjugate the transversal beam so that the product
of the two parabolic quasimodes contains \(|v_s|^2\).

For a smooth \(v=v(x')\), the product structure and
\cref{eq:L-plus,eq:L-minus} give
\begin{align}
 \cL_{+,h}(\zeta_+e^{i\eta x_1}v)
 &=e^{i\eta x_1}\zeta_+
   (-h^2\Delta_{g_0}-\alpha^2+h^2\eta^2-2ih\eta)v
   +h^2e^{i\eta x_1}(\zeta_+'+q\zeta_+)v,
 \label{eq:Lplus-on-ansatz}\\
 \cL_{-,h}(\zeta_-e^{i\eta x_1}\overline v)
 &=e^{i\eta x_1}\zeta_-
   (-h^2\Delta_{g_0}-\alpha^2+h^2\eta^2+2ih\eta)\overline v
   +h^2e^{i\eta x_1}(-\zeta_-'+\overline q\,\zeta_-)\overline v.
 \label{eq:Lminus-on-ansatz}
\end{align}
By \eqref{eq:s-eta},
\begin{align}
 \alpha^2(h^2s^2-1)+h^2\eta^2-2ih\eta
 &=-\alpha^2\beta^2\sigma^2h^2,
 \label{eq:algebra-plus}\\
 \alpha^2(h^2\overline s^{\,2}-1)+h^2\eta^2+2ih\eta
 &=-\alpha^2\beta^2\sigma^2h^2.
 \label{eq:algebra-minus}
\end{align}
Consequently,
\begin{align}
 \cL_{+,h}(\zeta_+e^{i\eta x_1}v)
 &=\zeta_+e^{i\eta x_1}P_{\perp,s}v
 +h^2e^{i\eta x_1}
   [\zeta_+'+(q-\alpha^2\beta^2\sigma^2)\zeta_+]v,
 \label{eq:reduced-plus}\\
 \cL_{-,h}(\zeta_-e^{i\eta x_1}\overline v)
 &=\zeta_-e^{i\eta x_1}\overline{P_{\perp,s}v}
 +h^2e^{i\eta x_1}
   [-\zeta_-'+(\overline q-\alpha^2\beta^2\sigma^2)\zeta_-]\overline v.
 \label{eq:reduced-minus}
\end{align}
Since \(M\Subset\R\times M_0^{\mathrm{int}}\), its \(x_1\)-projection is
bounded.  Taking \(v=v_s\) in the last two identities and using
\eqref{eq:transversal-quasimode-estimates} gives
\begin{align}
 \|V_{\pm,h}\|_{L^2(Q)}&\leq C,\notag\\
 \|\cL_{\pm,h}V_{\pm,h}\|_{L^2(Q)}
 &\leq C\bigl(h^{3/2}+h^2\bigr)=O(h^{3/2}),
 \label{eq:parabolic-residual-from-transversal}
\end{align}
which proves \eqref{eq:quasimode-estimates}.  Since \(\zeta_\pm\) were
arbitrary fixed smooth functions, the same construction is available with any
prescribed time amplitudes needed below.
\end{proof}

\begin{proof}[Proof of \cref{prop:quasimode-concentration}]
Use the normalized beam from \cref{lem:transversal-concentration} in the
construction of \cref{prop:parabolic-quasimodes}.  The calculation above
already gives \eqref{eq:quasimode-estimates}.  Moreover,
\begin{equation}
 V_{+,h}V_{-,h}
 =\zeta_+(t)\zeta_-(t)e^{2i\alpha^2\sigma x_1}|v_s(x')|^2.
 \label{eq:quasimode-product}
\end{equation}
For each fixed \((t,x_1)\), apply \eqref{eq:transversal-concentration} with
\(\psi(x')=F(t,x_1,x')\).  The bound
\begin{equation}
 \left|\int_{M_0}F(t,x_1,x')|v_s(x')|^2\,dV_{g_0}(x')\right|
 \leq C\|F(t,x_1,\cdot)\|_{L^\infty(M_0)}
 \label{eq:concentration-dominated-bound}
\end{equation}
follows from \eqref{eq:global-beam-L2-bound}.  Since \(F\) is compactly
supported in \(x_1\), dominated convergence in \((t,x_1)\) now gives
\eqref{eq:quasimode-concentration}.
\end{proof}

\section{Construction of geometric optics solutions}
\label{sec:go-solutions}

We first derive solvability statements that permit prescribed Dirichlet values on
one strict side of the boundary.  We then use them to correct the quasimodes from
\cref{sec:gaussian-beams}, obtaining exact solutions whose lateral traces are
supported in \(U\) and \(V\), respectively.

\subsection{Solvability with prescribed partial boundary values}

For \(\varepsilon>0\), set
\begin{equation}
 \Sigma_{+,\varepsilon}
 :=\{(t,x)\in\Sigma:\partial_\nu\phi(x)\geq\varepsilon\},
 \qquad
 \Sigma_{-,\varepsilon}
 :=\{(t,x)\in\Sigma:\partial_\nu\phi(x)\leq-\varepsilon\}.
 \label{eq:sigma-epsilon}
\end{equation}
Whenever an expression such as \(\{\partial_\nu\phi>0\}\) or
\(\{\partial_\nu\phi<0\}\) is used below without its ambient set being
displayed, it denotes the corresponding subset of \(\Sigma\); for example,
\[
 \{\partial_\nu\phi>0\}
 :=\{(t,x)\in\Sigma:\partial_\nu\phi(x)>0\}.
\]

\begin{lemma}[Normal-trace density]\label{lem:normal-trace-density}
Let
\[
 \cV=\{v\in C^\infty(\overline Q):
 v|_\Sigma=0,\ v(0,\cdot)=v(T,\cdot)=0\}.
\]
Then \(\{\partial_\nu v|_\Sigma:v\in\cV\}\) contains
\(C_0^\infty(\Sigma)\), and hence is dense both in
\(H^{1/2,1/4}(\Sigma)\) and in \(L^2(\Sigma)\).
\end{lemma}

\begin{proof}
Choose a boundary collar and a smooth defining function \(r\) with
\(r|_{\partial M}=0\) and \(\partial_\nu r=1\) on \(\partial M\).  Given
\(\eta\in C_0^\infty(\Sigma)\), extend \(\eta\) smoothly in the collar,
independently of \(r\) near the boundary, and multiply \(r\eta\) by a cutoff
supported in the collar.  The resulting function belongs to \(\cV\) and has
normal derivative \(\eta\) on \(\Sigma\).  Finally,
\(C_0^\infty(\Sigma)\) is dense in \(H^{1/2,1/4}(\Sigma)\) because the
temporal exponent \(1/4\) is below the endpoint trace threshold, and it is
also dense in \(L^2(\Sigma)\).
\end{proof}

\begin{proposition}\label{prop:partial-boundary-solvability}
Fix \(\varepsilon>0\) and let \(q\in L^\infty(Q)\).  There are \(C>0\) and
\(h_0>0\) such that the following statements hold for \(0<h\leq h_0\).
\begin{enumerate}[label=\textup{(\roman*)}]
\item For \(F\in L^2(Q)\) and
\(f_-\in L^2(\Sigma_{-,\varepsilon})\), there exists \(R\in L^2(Q)\) such
that \(e^{\Phi_h}R\in\cH_+(Q)\), \(\tau_0(e^{\Phi_h}R)\in L^2(\Sigma)\), and
\begin{align}
 e^{-\Phi_h}P_q(e^{\Phi_h}R)&=F&&\text{in }Q,
 \notag\\
 \tau_0(e^{\Phi_h}R)|_{\Sigma_{-,\varepsilon}}
 &=e^{\Phi_h}f_-&&\text{on }\Sigma_{-,\varepsilon}.
 \label{eq:solvability-forward}
\end{align}
Moreover,
\begin{equation}
 \|R\|_{L^2(Q)}
 \leq C\left(h\|F\|_{L^2(Q)}
 +\left(\frac h\varepsilon\right)^{1/2}
 \|f_-\|_{L^2(\Sigma_{-,\varepsilon})}\right).
 \label{eq:solvability-forward-estimate}
\end{equation}

\item For \(F\in L^2(Q)\) and
\(f_+\in L^2(\Sigma_{+,\varepsilon})\), there exists \(R\in L^2(Q)\) such
that \(e^{-\Phi_h}R\in\cH_-(Q)\), \(\tau_0(e^{-\Phi_h}R)\in L^2(\Sigma)\),
and
\begin{align}
 e^{\Phi_h}P_q^*(e^{-\Phi_h}R)&=F&&\text{in }Q,
 \notag\\
 \tau_0(e^{-\Phi_h}R)|_{\Sigma_{+,\varepsilon}}
 &=e^{-\Phi_h}f_+&&\text{on }\Sigma_{+,\varepsilon}.
 \label{eq:solvability-adjoint}
\end{align}
Moreover,
\begin{equation}
 \|R\|_{L^2(Q)}
 \leq C\left(h\|F\|_{L^2(Q)}
 +\left(\frac h\varepsilon\right)^{1/2}
 \|f_+\|_{L^2(\Sigma_{+,\varepsilon})}\right).
 \label{eq:solvability-adjoint-estimate}
\end{equation}
\end{enumerate}
\end{proposition}

\begin{proof}
We prove (i) first.  Let
\[
 \cA_h:=e^{\Phi_h}P_q^*e^{-\Phi_h}
\]
and let \(\cV\) be the test space introduced in
\cref{lem:normal-trace-density}.
The adjoint Carleman estimate only requires \(v(T,\cdot)=0\).  The additional
condition \(v(0,\cdot)=0\) is imposed for the duality argument and removes the
remaining temporal trace term; restricting to \(\cV\) is therefore compatible
with the estimate.
Divide the adjoint Carleman estimate \eqref{eq:carleman-adjoint} by \(h^4\).
Its favorable boundary term and
\(|\partial_\nu\phi|\geq\varepsilon\) on
\(\Sigma_{-,\varepsilon}\) give
\begin{align}
 C\left(\frac1h\|v\|_{L^2(Q)}
 +\left(\frac\varepsilon h\right)^{1/2}
 \|\partial_\nu v\|_{L^2(\Sigma_{-,\varepsilon})}\right)
 &\leq\|\cA_hv\|_{L^2(Q)}\notag\\
 &\quad+h^{-1/2}
 \bigl\||\partial_\nu\phi|^{1/2}\partial_\nu v\bigr\|_
 {L^2(\{\partial_\nu\phi>0\})}
 \label{eq:duality-carleman}
\end{align}
for \(v\in\cV\).

Let
\[
 \cY:=L^2(Q)\times L^2(\{\partial_\nu\phi>0\})
\]
with norm
\[
 \|(w,w_+)\|_{\cY}^2
 :=\|w\|_{L^2(Q)}^2+h^{-1}\|w_+\|_
 {L^2(\{\partial_\nu\phi>0\})}^2,
\]
and define
\[
 \cX:=\left\{
 \left(\cA_hv,
 |\partial_\nu\phi|^{1/2}\partial_\nu v
 \big|_{\{\partial_\nu\phi>0\}}\right):v\in\cV\right\}\subset\cY.
\]
Estimate \eqref{eq:duality-carleman} makes the following functional on \(\cX\)
well-defined:
\begin{align*}
 \ell\left(\cA_hv,
 |\partial_\nu\phi|^{1/2}\partial_\nu v
 \big|_{\{\partial_\nu\phi>0\}}\right)
 :=(F,v)_{L^2(Q)}
 -(f_-,\partial_\nu v)_{L^2(\Sigma_{-,\varepsilon})}.
\end{align*}
Indeed,
\begin{align*}
 |\ell|
 \leq C\left(h\|F\|_{L^2(Q)}
 +\left(\frac h\varepsilon\right)^{1/2}
 \|f_-\|_{L^2(\Sigma_{-,\varepsilon})}\right)
 \left\|\left(\cA_hv,
 |\partial_\nu\phi|^{1/2}\partial_\nu v\right)\right\|_{\cY}.
\end{align*}
Hahn--Banach and Riesz representation yield
\(R\in L^2(Q)\) and
\(S_+\in L^2(\{\partial_\nu\phi>0\})\).  Since the second component of
\(\cY\) carries the weight \(h^{-1}\), the Riesz representation is
\begin{align*}
 &(R,\cA_hv)_{L^2(Q)}
 +h^{-1}(S_+,|\partial_\nu\phi|^{1/2}\partial_\nu v)_
 {L^2(\{\partial_\nu\phi>0\})}\\
 &\qquad=(F,v)_{L^2(Q)}
 -(f_-,\partial_\nu v)_{L^2(\Sigma_{-,\varepsilon})},
 \qquad v\in\cV.
\end{align*}
The norm of the Riesz vector is bounded by the right-hand side of the
preceding estimate.  In particular, its first component \(R\) satisfies
\eqref{eq:solvability-forward-estimate}.  Set \(R_+:=h^{-1}S_+\).  Then
\begin{align}
 &(R,\cA_hv)_{L^2(Q)}
 +(R_+,|\partial_\nu\phi|^{1/2}\partial_\nu v)_
 {L^2(\{\partial_\nu\phi>0\})}\notag\\
 &\qquad=(F,v)_{L^2(Q)}
 -(f_-,\partial_\nu v)_{L^2(\Sigma_{-,\varepsilon})},
 \qquad v\in\cV,
 \label{eq:riesz-solvability}
\end{align}

Taking \(v\in C_0^\infty(Q)\) in \eqref{eq:riesz-solvability} gives
\[
 e^{-\Phi_h}P_q(e^{\Phi_h}R)=F
\]
in distributions.  Put \(W=e^{\Phi_h}R\).  For fixed \(h\), the exponential is
smooth and bounded on \(\overline Q\).  Since \(P_qW=e^{\Phi_h}F\in L^2(Q)\),
we have \(W\in\cH_+(Q)\), so its generalized Dirichlet trace is defined.

Apply the generalized Green identity from \cref{app:well-posedness} to \(W\)
and \(e^{-\Phi_h}v\).  The equation and the vanishing temporal traces of \(v\)
give
\[
 (R,\cA_hv)_{L^2(Q)}
 =(F,v)_{L^2(Q)}
 -(\tau_0W,e^{-\Phi_h}\partial_\nu v)_
 {H^{-1/2,-1/4}(\Sigma),H^{1/2,1/4}(\Sigma)}.
\]
For each fixed \(h\), multiplication by \(e^{-\Phi_h}\) is a smooth
invertible multiplier on the lateral trace spaces.  Hence
\cref{lem:normal-trace-density} also gives density for the weighted normal
traces \(e^{-\Phi_h}\partial_\nu v\).
Comparison with \eqref{eq:riesz-solvability} shows that \(\tau_0W\) is
represented by the \(L^2(\Sigma)\) function
\[
 \Gamma_h(t,x)=
 \begin{cases}
  e^{\Phi_h}f_-(t,x),&(t,x)\in\Sigma_{-,\varepsilon},\\
  e^{\Phi_h}|\partial_\nu\phi|^{1/2}R_+(t,x),
     &\partial_\nu\phi(x)>0,\\
 0,&\text{otherwise}.
 \end{cases}
\]
By this density, comparison against all normal traces identifies the
generalized trace \(\tau_0W\) with \(\Gamma_h\).  In
particular, \(\tau_0W\in L^2(\Sigma)\), and its restriction to
\(\Sigma_{-,\varepsilon}\) is the prescribed trace.  This proves (i).

For (ii), repeat the same duality argument with
\(\mathcal B_h=e^{-\Phi_h}P_qe^{\Phi_h}\) and the forward Carleman estimate
\eqref{eq:carleman-forward}.  Indeed, since \(\Phi_h\) is real,
\(\mathcal B_h\) is the formal \(L^2(Q)\)-adjoint of
\(e^{\Phi_h}P_q^*e^{-\Phi_h}\), where \(P_q^*\) is the formal adjoint defined
in \eqref{eq:operators-adjoint-transpose}.  The favorable boundary term in the
forward estimate lies on \(\{\partial_\nu\phi>0\}\); in the duality argument
this is precisely the side on which the trace \(f_+\) is prescribed, while the
Riesz boundary component lies on the opposite side.  Thus the prescribed trace
is imposed on \(\Sigma_{+,\varepsilon}\).  The rest of the argument is exactly
as above and gives \eqref{eq:solvability-adjoint} and
\eqref{eq:solvability-adjoint-estimate} with the same powers of \(h\) and
\(\varepsilon\), proving (ii).
\end{proof}

\subsection{Boundary bounds for the quasimodes}

\begin{lemma}\label{lem:quasimode-boundary-bound}
For every fixed \(\varepsilon>0\) and fixed
\(\zeta_\pm\in C^\infty([0,T])\), the quasimodes chosen as in
\cref{prop:quasimode-concentration} satisfy
\begin{equation}
 \|V_{+,h}\|_{L^2(\Sigma_{-,\varepsilon})}
 +\|V_{-,h}\|_{L^2(\Sigma_{+,\varepsilon})}=O(1).
 \label{eq:quasimode-boundary-bound}
\end{equation}
The estimate is uniform for \(\sigma\) in a fixed compact subset of \(\R\).
\end{lemma}

\begin{proof}
On the part of \(\partial M\) where
\(|\partial_\nu x_1|\geq\varepsilon\), the projection
\[
 \pi:\partial M\to M_0,\qquad \pi(x_1,x')=x',
\]
is a local diffeomorphism.  Compactness gives a finite cover by boundary graphs
\(x_1=f_k(x')\), whose surface Jacobians are bounded in terms of
\(\varepsilon\).  By \eqref{eq:parabolic-quasimode-ansatz},
\[
 V_{+,h}=\zeta_+(t)e^{i\alpha^2\sigma x_1}v_s(x'),\qquad
 V_{-,h}=\zeta_-(t)e^{i\alpha^2\sigma x_1}\overline{v_s(x')}.
\]
The result follows by integrating over the finitely many graphs and using
\(\|v_s\|_{L^2(M_0)}=O(1)\).
\end{proof}

\subsection{Exact solutions with prescribed lateral supports}

\begin{proposition}\label{prop:exact-go-solutions}
Let \(q\in L^\infty(Q)\), let \(\gamma:[0,L]\to M_0\) be a unit-speed
non-tangential geodesic, and fix
\(\sigma\in\R\) and \(\zeta_\pm\in C^\infty([0,T])\).  There exist families
\begin{equation}
 u_{+,h}=e^{\Phi_h}(V_{+,h}+R_{+,h}),\qquad
 u_{-,h}=e^{-\Phi_h}(V_{-,h}+R_{-,h})
 \label{eq:go-solutions}
\end{equation}
such that
\begin{align}
 P_qu_{+,h}&=0\quad\text{in }Q,&
 \supp(\tau_0u_{+,h})&\subset U,
 \label{eq:go-forward}\\
 P_q^{\mathrm t}u_{-,h}&=0\quad\text{in }Q,&
 \supp(\tau_0u_{-,h})&\subset V.
 \label{eq:go-backward}
\end{align}
Moreover,
\[
 u_{+,h}\in\cH_+(Q),\qquad
 u_{-,h}\in\cH_-(Q),\qquad
 \tau_0u_{\pm,h}\in L^2(\Sigma),
\]
and
\begin{equation}
 \|R_{+,h}\|_{L^2(Q)}+\|R_{-,h}\|_{L^2(Q)}=O(h^{1/2}).
 \label{eq:go-remainder}
\end{equation}
Here \(V_{+,h}\) is constructed for \(q\).  For \(V_{-,h}\), the construction
in \cref{sec:gaussian-beams} is applied to the coefficient \(\overline q\), so
that the adjoint there is \(P_q^{\mathrm t}=P_{\overline q}^*\).  This does not
alter the concentration formula.
\end{proposition}

\begin{proof}
Choose open neighborhoods \(U_0'\), \(V_0'\) of \(\partial M_+\),
\(\partial M_-\), respectively, such that
\[
 \overline{U_0'}\subset U_0,\qquad
 \overline{V_0'}\subset V_0,
\]
and put \(U'=(0,T)\times U_0'\), \(V'=(0,T)\times V_0'\).
Compactness gives \(\varepsilon>0\) such that
\begin{equation}
 \Sigma\setminus U'\subset\Sigma_{-,\varepsilon},\qquad
 \Sigma\setminus V'\subset\Sigma_{+,\varepsilon}.
 \label{eq:UV-epsilon-inclusions}
\end{equation}

For the growing solution, set
\[
 F_{+,h}:=-e^{-\Phi_h}P_q(e^{\Phi_h}V_{+,h})
 =-h^{-2}\cL_{+,h}V_{+,h}.
\]
Then \(\|F_{+,h}\|_{L^2(Q)}=O(h^{-1/2})\).  On
\(\Sigma_{-,\varepsilon}\), define
\[
 f_{+,h}:=
 \begin{cases}
  -V_{+,h},&(t,x)\in\Sigma\setminus U',\\
  0,&(t,x)\in\Sigma_{-,\varepsilon}\cap U'.
 \end{cases}
\]
By \cref{lem:quasimode-boundary-bound},
\(\|f_{+,h}\|_{L^2(\Sigma_{-,\varepsilon})}=O(1)\).
Apply part (i) of \cref{prop:partial-boundary-solvability}.  The correction
\(R_{+,h}\) cancels \(V_{+,h}\) on \(\Sigma\setminus U'\).  Indeed, on this
set the prescribed trace gives
\[
 \tau_0u_{+,h}
 =e^{\Phi_h}V_{+,h}+\tau_0(e^{\Phi_h}R_{+,h})
 =e^{\Phi_h}(V_{+,h}+f_{+,h})=0.
\]
Hence
\[
 \supp(\tau_0u_{+,h})\subset\overline{U'}\subset U.
\]
Moreover,
\[
 \|R_{+,h}\|_{L^2(Q)}
 \leq C\left(h\,O(h^{-1/2})
 +\left(\frac h\varepsilon\right)^{1/2}O(1)\right)
 =O(h^{1/2}).
\]

For the decaying solution, let
\[
 \cL_{-,h}^{\mathrm t}
 :=h^2e^{\Phi_h}P_q^{\mathrm t}e^{-\Phi_h}
 =h^2e^{\Phi_h}P_{\overline q}^*e^{-\Phi_h},
\]
and set
\[
 F_{-,h}:=-e^{\Phi_h}P_q^{\mathrm t}(e^{-\Phi_h}V_{-,h})
 =-h^{-2}\cL_{-,h}^{\mathrm t}V_{-,h}.
\]
Again \(\|F_{-,h}\|_{L^2(Q)}=O(h^{-1/2})\).  Define on
\(\Sigma_{+,\varepsilon}\)
\[
 f_{-,h}:=
 \begin{cases}
  -V_{-,h},&(t,x)\in\Sigma\setminus V',\\
  0,&(t,x)\in\Sigma_{+,\varepsilon}\cap V'.
 \end{cases}
\]
Its \(L^2(\Sigma_{+,\varepsilon})\)-norm is \(O(1)\).  Apply part (ii) of
\cref{prop:partial-boundary-solvability} with coefficient \(\overline q\).
On \(\Sigma\setminus V'\), the prescribed trace gives
\[
 \tau_0u_{-,h}
 =e^{-\Phi_h}V_{-,h}+\tau_0(e^{-\Phi_h}R_{-,h})
 =e^{-\Phi_h}(V_{-,h}+f_{-,h})=0.
\]
The correction also satisfies
\[
 e^{\Phi_h}P_q^{\mathrm t}
 \bigl(e^{-\Phi_h}(V_{-,h}+R_{-,h})\bigr)=0,
\qquad
 \supp(\tau_0u_{-,h})\subset\overline{V'}\subset V,
\]
and \(\|R_{-,h}\|_{L^2(Q)}=O(h^{1/2})\).  The remaining assertions follow
from the solvability proposition.
\end{proof}

\begin{remark}\label{rem:beam-order-sufficient}
The \(O(h^{3/2})\) quasimode residual in
\eqref{eq:quasimode-estimates} is sufficient here.  After division by \(h^2\)
it produces a source of size \(O(h^{-1/2})\), and the boundary solvability
estimate gives the \(O(h^{1/2})\) correction in \eqref{eq:go-remainder}.
In the uniqueness argument the lateral boundary term vanishes exactly by the
support conditions in \(U\) and \(V\); it is not estimated through a trace of
the correction.  Thus no higher-order transversal quasimode is needed for the
present partial input--output formulation.
\end{remark}

\section{Proof of the uniqueness result}
\label{sec:uniqueness}

By \cref{prop:map-reduction}, it is enough to work with the product metric
\eqref{eq:product-metric-standing}.  Let \(q_1,q_2\in C(\overline Q;\C)\) satisfy
\begin{equation}
 q_1=q_2\quad\text{on }\Sigma,\qquad
 \Lambda_{g,q_1}^{U,V}=\Lambda_{g,q_2}^{U,V},
 \label{eq:measurement-equality}
\end{equation}
and put \(p=q_2-q_1\).

\subsection{The integral identity}

Fix the parameters in \cref{eq:beta-alpha}, a unit-speed non-tangential geodesic
\(\gamma:[0,L]\to M_0\), \(\sigma\in\R\), and
\(\zeta_\pm\in C^\infty([0,T])\).  Let
\[
 u_{+,h}=e^{\Phi_h}(V_{+,h}+R_{+,h})
\]
be the growing solution from \cref{prop:exact-go-solutions} for \(q_1\).  Thus
\begin{equation}
 P_{q_1}u_{+,h}=0,\qquad
 \supp(\tau_0u_{+,h})\subset U,\qquad
 \tau_0u_{+,h}\in L^2(\Sigma).
 \label{eq:forward-go-uniqueness}
\end{equation}
Use \(r_0u_{+,h}\) and \(\tau_0u_{+,h}\) as the initial and Dirichlet inputs
for the equation with potential \(q_2\), and denote its transposition solution by
\(v_h\).  Equality of the partial maps gives
\begin{equation}
 r_Tv_h=r_Tu_{+,h},\qquad
 \tau_1v_h|_V=\tau_1u_{+,h}|_V.
 \label{eq:equal-output-traces}
\end{equation}

Set \(w_h=u_{+,h}-v_h\).  Then
\begin{align}
 P_{q_2}w_h&=p\,u_{+,h}&&\text{in }Q,\notag\\
 r_0w_h&=0,&
 \tau_0w_h&=0\quad\text{on }\Sigma.
 \label{eq:w-equation}
\end{align}
These are the initial and lateral conditions coming from using the same input.
Separately, the equality of the measured outputs in
\eqref{eq:equal-output-traces} gives the additional relations
\begin{equation}
 r_Tw_h=0,\qquad \tau_1w_h|_V=0.
 \label{eq:w-output-zero}
\end{equation}

For fixed \(h\), the source \(p u_{+,h}\) belongs to \(L^2(Q)\).  At this
point \(w_h\) is a difference of transposition solutions, so we justify the
strong regularity used below.  Let \(\widetilde w_h\) be the unique strong
solution of
\[
 P_{q_2}\widetilde w_h=p u_{+,h},\qquad
 r_0\widetilde w_h=0,\qquad \tau_0\widetilde w_h=0.
\]
Standard \(L^2\) parabolic regularity gives
\(\widetilde w_h\in H^{2,1}(Q)\).  Both \(w_h\) and
\(\widetilde w_h\) solve the same initial--boundary value problem in the
transposition sense, and uniqueness in
\cref{prop:appendix-reduced-transposition} therefore gives
\(w_h=\widetilde w_h\).  Consequently,
\begin{equation}
 w_h\in H^{2,1}(Q),\qquad
 \tau_1w_h\in H^{1/2,1/4}(\Sigma).
 \label{eq:w-regularity}
\end{equation}

Take the decaying solution for \(q_2\),
\[
 u_{-,h}=e^{-\Phi_h}(V_{-,h}+R_{-,h}).
\]
It satisfies
\begin{equation}
 P_{q_2}^{\mathrm t}u_{-,h}=0,\qquad
 \supp(\tau_0u_{-,h})\subset V,\qquad
 \tau_0u_{-,h}\in L^2(\Sigma).
 \label{eq:backward-go-uniqueness}
\end{equation}
We use the bilinear Green identity associated with the formal transpose
\(P_q^{\mathrm t}\).  For \(w,z\in C^\infty(\overline Q)\), and by density whenever
the displayed traces are defined,
\begin{align}
 &\int_Q(P_qw)z\,dV_gdt-\int_Qw(P_q^{\mathrm t}z)\,dV_gdt\notag\\
 &\quad=\int_M\bigl[(r_Tw)(r_Tz)-(r_0w)(r_0z)\bigr]\,dV_g\notag\\
 &\qquad\quad+\int_\Sigma
 \bigl[(\tau_0w)(\tau_1z)-(\tau_1w)(\tau_0z)\bigr]\,dS_gdt.
 \label{eq:bilinear-green}
\end{align}
For the present pair \((w_h,u_{-,h})\), the identity follows by graph-norm
approximation of \(u_{-,h}\); the terms that are not classical are understood
by duality.  Equations \eqref{eq:w-equation} and \eqref{eq:w-output-zero},
together with \eqref{eq:backward-go-uniqueness}, reduce the right-hand side to
\[
 -\int_\Sigma(\tau_1w_h)(\tau_0u_{-,h})\,dS_gdt.
\]
This is zero because the first factor vanishes on \(V\) and the second is
supported in \(V\).  Hence
\begin{equation}
 \int_Qp(t,x)u_{+,h}(t,x)u_{-,h}(t,x)\,dV_gdt=0.
 \label{eq:integral-identity}
\end{equation}
This is precisely where the two partial boundary regions are used: the growing
trace is an admissible input supported in \(U\), equality of the measurements
gives the vanishing of \(\tau_1w_h\) on \(V\), and the decaying trace is
supported in \(V\).  No solution with unrestricted lateral trace is needed.

\subsection{Passage to the Gaussian beam limit}

The exponential factors in \eqref{eq:integral-identity} cancel, giving
\begin{equation}
 0=\int_Qp\,(V_{+,h}+R_{+,h})(V_{-,h}+R_{-,h})\,dV_gdt.
 \label{eq:expanded-integral-identity}
\end{equation}
By \cref{eq:quasimode-estimates,eq:go-remainder},
\[
 \|V_{\pm,h}\|_{L^2(Q)}=O(1),\qquad
 \|R_{\pm,h}\|_{L^2(Q)}=O(h^{1/2}).
\]
The three terms in \eqref{eq:expanded-integral-identity} containing a remainder
are respectively \(O(h^{1/2})\), \(O(h^{1/2})\), and \(O(h)\).  Therefore
\begin{equation}
 \lim_{h\to0}\int_Qp(t,x)V_{+,h}(t,x)V_{-,h}(t,x)\,dV_gdt=0.
 \label{eq:leading-quasimode-product}
\end{equation}

Extend \(p(t,\cdot)\) by zero from \(M\) to \(\R\times M_0\), and denote the
extension by \(\widetilde p\).  By \cref{rem:boundary-agreement}, this extension
is continuous, and it is compactly supported in \(x_1\).  Applying
\cref{prop:quasimode-concentration} to
\eqref{eq:leading-quasimode-product} yields
\begin{equation}
 \int_0^T\int_\R\int_0^L
 \widetilde p(t,x_1,\gamma(r))\zeta_+(t)\zeta_-(t)
 e^{2i\alpha^2\sigma x_1}e^{-2\alpha\sigma r}
 \,dr\,dx_1\,dt=0.
 \label{eq:gaussian-limit-identity}
\end{equation}
Choose \(\zeta_-\equiv1\) and vary \(\zeta_+\in C_0^\infty(0,T)\).  For fixed
\(\gamma\) and \(\sigma\), the fundamental lemma of the calculus of variations
first gives the following identity for every \(t\in(0,T)\); continuity in
\(t\) then extends it to \(t=0,T\):
\begin{equation}
 \int_\R\int_0^L
 \widetilde p(t,x_1,\gamma(r))
 e^{2i\alpha^2\sigma x_1}e^{-2\alpha\sigma r}
 \,dr\,dx_1=0
 \label{eq:time-localized-identity}
\end{equation}
for every \(t\in[0,T]\).

Define the partial Fourier transform
\[
 \widehat p(t,\xi,x')
 :=\int_\R e^{-i\xi x_1}\widetilde p(t,x_1,x')\,dx_1.
\]
Equation \eqref{eq:time-localized-identity} becomes
\begin{equation}
 I_{-2\alpha\sigma}
 \bigl(\widehat p(t,-2\alpha^2\sigma,\cdot)\bigr)(\gamma)=0.
 \label{eq:attenuated-ray-identity}
\end{equation}
This holds for every non-tangential unit-speed geodesic \(\gamma\).  For all
sufficiently small \(|\sigma|\), one has
\(|2\alpha\sigma|<\varepsilon_0\), so
\cref{ass:ray-injectivity} implies
\begin{equation}
 \widehat p(t,-2\alpha^2\sigma,x')=0,
 \qquad(t,x')\in[0,T]\times M_0.
 \label{eq:fourier-local-vanishing}
\end{equation}
For fixed \((t,x')\), the function
\(\xi\mapsto\widehat p(t,\xi,x')\) is entire because \(\widetilde p\) is
compactly supported in \(x_1\).  It vanishes on an open interval of real
frequencies by \eqref{eq:fourier-local-vanishing}; analytic continuation and
Fourier inversion give \(\widetilde p=0\).  Thus \(q_1=q_2\) in the product
setting.  Finally, \eqref{eq:potential-difference-transform} gives the same
conclusion for the original CTA metric, proving \cref{thm:main}.

\appendix

\section{Well-posedness of the direct problem and the measurement map}
\label{app:well-posedness}

We first prove the trace and transposition results after the reduction \(c=1\),
using the graph spaces \(\cH_\pm(Q)\) already defined in
\eqref{eq:reduced-graph-spaces}.  We then transfer the result to the original CTA
operator.  The transfer is based only on a graph-space isomorphism and therefore
does not presuppose well-posedness for the original problem.  We use the
standard \(L^2\) parabolic regularity, trace-extension, and transposition
framework; see, for example, Lions and Magenes \cite{LionsMagenes1972} and the
anisotropic trace-space formulation in
\cite[Appendix~A]{ChoulliKian2018}.  The details needed for the present graph
spaces are included below to fix the trace spaces and the input--output map
unambiguously.

Set
\[
 H^{2,1}(Q):=L^2(0,T;H^2(M))\cap H^1(0,T;L^2(M)).
\]
The parabolic trace-extension theorem will be used in the following form.  Given
\(\eta\in H^{1/2,1/4}(\Sigma)\), there is \(Z_\eta\in H^{2,1}(Q)\) such that
\[
 Z_\eta|_\Sigma=0,\qquad
 \partial_\nu Z_\eta|_\Sigma=\eta,\qquad
 Z_\eta|_{t=0}=Z_\eta|_{t=T}=0,
\]
and \(\|Z_\eta\|_{H^{2,1}(Q)}
\leq C\|\eta\|_{H^{1/2,1/4}(\Sigma)}\).  Given
\(\zeta\in H_0^{3/2,3/4}(\Sigma)\), there is \(Z_\zeta\in H^{2,1}(Q)\) with
\[
 Z_\zeta|_\Sigma=\zeta,\qquad
 \partial_\nu Z_\zeta|_\Sigma=0,\qquad
 Z_\zeta|_{t=0}=Z_\zeta|_{t=T}=0,
\]
and the analogous estimate.  We also use the corresponding extension of
compatible temporal traces.  Notice that
\(H_0^{1/2,1/4}(\Sigma)=H^{1/2,1/4}(\Sigma)\), since the temporal exponent is
below the endpoint trace threshold.

\subsection{Reduced graph spaces and generalized traces}

For \(u\in C^\infty(\overline Q)\), write
\[
 \tau_0u=u|_\Sigma,\qquad
 \tau_1u=\partial_\nu u|_\Sigma,\qquad
 r_0u=u(0,\cdot),\qquad r_Tu=u(T,\cdot).
\]
These are the classical lateral Dirichlet and Neumann traces and the two
temporal endpoint traces.  The purpose of the next proposition is to record
their continuous extensions when only \(u\) and the corresponding parabolic
operator applied to \(u\) are in \(L^2(Q)\).  These are the extensions used in
the input--output map and in the Green identities of \cref{sec:go-solutions}.

\begin{proposition}\label{prop:appendix-reduced-traces}
The preceding maps extend uniquely to bounded operators
\[
 \begin{aligned}
 \tau_0&:\cH_\pm(Q)\to H^{-1/2,-1/4}(\Sigma),&
 \tau_1&:\cH_\pm(Q)\to H^{-3/2,-3/4}(\Sigma),\\
 r_0,r_T&:\cH_\pm(Q)\to H^{-1}(M).
 \end{aligned}
\]
Moreover,
\begin{align}
 &\|\tau_0u\|_{H^{-1/2,-1/4}(\Sigma)}
 +\|\tau_1u\|_{H^{-3/2,-3/4}(\Sigma)}
 +\|r_0u\|_{H^{-1}(M)}+\|r_Tu\|_{H^{-1}(M)}
 \leq C\|u\|_{\cH_\pm(Q)}.
 \label{eq:reduced-trace-estimate}
\end{align}
\end{proposition}

\begin{proof}
We treat \(\cH_+(Q)\); the other case follows by reversing time.  For
\(u,Z\in C^\infty(\overline Q)\), integration by parts gives
\begin{align}
 &((\partial_t-\Delta_g)u,Z)_{L^2(Q)}
 -(u,(-\partial_t-\Delta_g)Z)_{L^2(Q)}\notag\\
 &\quad=(\tau_0u,\partial_\nu Z)_\Sigma
 -(\tau_1u,Z|_\Sigma)_\Sigma
 +(r_Tu,Z(T,\cdot))_M-(r_0u,Z(0,\cdot))_M.
 \label{eq:app-green-reduced}
\end{align}
Here and below, \((\cdot,\cdot)_\Sigma\) denotes the appropriate lateral
duality pairing, namely either the pairing between
\(H^{-1/2,-1/4}(\Sigma)\) and \(H^{1/2,1/4}(\Sigma)\), or the pairing between
\(H^{-3/2,-3/4}(\Sigma)\) and \(H_0^{3/2,3/4}(\Sigma)\).  Similarly,
\((\cdot,\cdot)_M\) denotes the pairing between \(H^{-1}(M)\) and
\(H_0^1(M)\).  For smooth arguments these pairings agree with the corresponding
\(L^2\) inner products under the convention fixed in
\cref{sec:preliminaries}.

Use \(Z_\eta\) in \eqref{eq:app-green-reduced}.  The only nonzero boundary
trace of \(Z_\eta\) is its normal derivative, so
\[
 (\tau_0u,\eta)_{H^{-1/2,-1/4},H^{1/2,1/4}}
 =((\partial_t-\Delta_g)u,Z_\eta)_{L^2(Q)}
 -(u,(-\partial_t-\Delta_g)Z_\eta)_{L^2(Q)}.
\]
This gives the required bound for \(\tau_0u\).  Similarly, the extension
\(Z_\zeta\) yields
\[
 -(\tau_1u,\zeta)_{H^{-3/2,-3/4},H_0^{3/2,3/4}}
 =((\partial_t-\Delta_g)u,Z_\zeta)_{L^2(Q)}
 -(u,(-\partial_t-\Delta_g)Z_\zeta)_{L^2(Q)},
\]
which gives the bound for \(\tau_1u\).

For \(\psi\in H_0^1(M)\), choose \(Z\in H^{2,1}(Q)\) such that
\[
 Z|_\Sigma=0,\qquad Z(0,\cdot)=0,\qquad Z(T,\cdot)=\psi,
 \qquad \|Z\|_{H^{2,1}(Q)}\leq C\|\psi\|_{H_0^1(M)}.
\]
Using the already extended \(\tau_0u\) in \eqref{eq:app-green-reduced} gives
\begin{align*}
 (r_Tu,\psi)_{H^{-1},H_0^1}
 ={}&((\partial_t-\Delta_g)u,Z)_{L^2(Q)}
 -(u,(-\partial_t-\Delta_g)Z)_{L^2(Q)}\\
 &-(\tau_0u,\partial_\nu Z)_
 {H^{-1/2,-1/4},H^{1/2,1/4}}.
\end{align*}
This proves the bound for \(r_Tu\); the initial trace is identical.  Finally,
the standard density theorem for maximal parabolic graph spaces gives density
of \(C^\infty(\overline Q)\) in \(\cH_\pm(Q)\); see, for example,
\cite{LionsMagenes1972}.  Hence the extensions are unique, completing the
proof.
\end{proof}

Identity \eqref{eq:app-green-reduced} extends by density to graph-space elements
whenever the second factor has the regularity and compatible traces required by
the displayed pairings.

\subsection{Transposition solutions for the reduced problem}

The preceding trace result allows us to formulate the direct problem for the
negative-order boundary and initial data used in the inverse problem.  We use
transposition because these data need not be classical traces of an
\(H^{2,1}(Q)\) solution; for smooth compatible data the definition agrees with
the usual parabolic solution.

\begin{proposition}\label{prop:appendix-reduced-transposition}
Let \(q\in L^\infty(Q)\).  For every
\[
 F\in L^2(Q),\qquad
 f\in H^{-1/2,-1/4}(\Sigma),\qquad
 u_0\in H^{-1}(M),
\]
there exists a unique \(u\in\cH_+(Q)\) satisfying
\begin{equation}
 P_qu=F\quad\text{in }Q,\qquad
 \tau_0u=f,\qquad r_0u=u_0
 \label{eq:reduced-forced-problem}
\end{equation}
in the transposition sense.  Moreover,
\begin{align}
 \|u\|_{\cH_+(Q)}
 &\leq C\bigl(\|F\|_{L^2(Q)}
 +\|f\|_{H^{-1/2,-1/4}(\Sigma)}
 +\|u_0\|_{H^{-1}(M)}\bigr),
 \label{eq:app-forward-estimate}\\
 \|r_Tu\|_{H^{-1}(M)}
 +\|\tau_1u\|_{H^{-3/2,-3/4}(\Sigma)}
 &\leq C\bigl(\|F\|_{L^2(Q)}
 +\|f\|_{H^{-1/2,-1/4}(\Sigma)}
 +\|u_0\|_{H^{-1}(M)}\bigr).
 \label{eq:app-output-estimate}
\end{align}
The analogous backward statement holds for \(P_q^*\), with \(r_Tu\)
prescribed instead of \(r_0u\).
\end{proposition}

\begin{proof}
For \(G\in L^2(Q)\), let \(z_G\in H^{2,1}(Q)\) be the unique solution of
\[
 P_q^*z_G=G\quad\text{in }Q,\qquad
 z_G|_\Sigma=0,\qquad z_G(T,\cdot)=0.
\]
Standard \(L^2\)-regularity for uniformly parabolic equations gives
\begin{equation}
 \|z_G\|_{H^{2,1}(Q)}
 +\|\partial_\nu z_G\|_{H^{1/2,1/4}(\Sigma)}
 +\|z_G(0,\cdot)\|_{H_0^1(M)}
 \leq C\|G\|_{L^2(Q)}.
 \label{eq:app-adjoint-regularity}
\end{equation}
Define
\begin{align}
 \ell(G):={}&(F,z_G)_{L^2(Q)}
 -(f,\partial_\nu z_G)_
 {H^{-1/2,-1/4}(\Sigma),H^{1/2,1/4}(\Sigma)}
 +(u_0,z_G(0,\cdot))_{H^{-1}(M),H_0^1(M)}.
 \label{eq:app-transposition-functional}
\end{align}
This is a bounded linear functional on \(L^2(Q)\).  Riesz representation gives
a unique \(u\in L^2(Q)\) such that
\begin{equation}
 (u,G)_{L^2(Q)}=\ell(G),\qquad G\in L^2(Q).
 \label{eq:app-transposition-definition}
\end{equation}
This identity defines the transposition solution.

Taking \(G=P_q^*\xi\) for \(\xi\in C_0^\infty(Q)\) shows that \(P_qu=F\) in
distributions.  Hence
\((\partial_t-\Delta_g)u=F-qu\in L^2(Q)\), so \(u\in\cH_+(Q)\).
If \(z\in H^{2,1}(Q)\) has zero lateral trace and vanishes at \(t=T\), use
\(G=P_q^*z\) in \eqref{eq:app-transposition-definition} and compare with
\eqref{eq:app-green-reduced}.  This yields
\begin{align}
 &(\tau_0u-f,\partial_\nu z)_
 {H^{-1/2,-1/4},H^{1/2,1/4}}
 -(r_0u-u_0,z(0,\cdot))_{H^{-1},H_0^1}=0.
 \label{eq:app-trace-identification}
\end{align}
The trace-extension theorem permits the two test traces to be prescribed
independently, and therefore \(\tau_0u=f\) and \(r_0u=u_0\).

The Riesz estimate, the equation, and
\cref{prop:appendix-reduced-traces} imply
\cref{eq:app-forward-estimate,eq:app-output-estimate}.  If all data vanish,
\eqref{eq:app-transposition-definition} gives \(u=0\), proving uniqueness.
Time reversal gives the backward result.
\end{proof}

Taking \(F=0\) and restricting the Neumann trace to \(V\) proves boundedness of
the reduced partial input--output map.

\subsection{Transfer to the original CTA operator}
\label{subsec:appendix-transfer}

We now recover the general-\(c\) statement without repeating the preceding
proof.  For the original metric \(g=c\widetilde g\), define
\[
 \cP_{c,\pm}:=\pm c^{-1}\partial_t-\Delta_g,\qquad
 \cH_{c,\pm}(Q)
 :=\{u\in L^2(Q):\cP_{c,\pm}u\in L^2(Q)\}.
\]
We equip these spaces with the graph norms
\[
 \|u\|_{\cH_{c,\pm}(Q)}^2
 :=\|u\|_{L^2(Q)}^2+\|\cP_{c,\pm}u\|_{L^2(Q)}^2.
\]
Thus \(\cH_{c,+}(Q)\) and \(\cH_{c,-}(Q)\) are simply the maximal \(L^2\)
graph spaces for the forward and backward principal parts before the conformal
reduction.  We introduce them here only to transfer the reduced trace and
well-posedness statements back to the original CTA operator.
Let \(\widetilde{\cH}_\pm(Q)\) be the spaces
\eqref{eq:reduced-graph-spaces} formed with \(\widetilde g\), and let
\[
 T_cu:=c^\rho u,\qquad \rho=\frac{n-2}{4}.
\]
Applying \eqref{eq:conformal-conjugation} with \(q=0\) shows
\begin{equation}
 (\pm\partial_t-\Delta_{\widetilde g})T_cu
 =c^{(n+2)/4}\cP_{c,\pm}u-q_cT_cu,\qquad
 q_c:=c^{-\rho}\Delta_{\widetilde g}(c^\rho).
 \label{eq:graph-space-conformal-identity}
\end{equation}
Since \(q_c\) and all powers of \(c\) are smooth and bounded, \(T_c\) is an
isomorphism
\begin{equation}
 T_c:\cH_{c,\pm}(Q)\longrightarrow\widetilde{\cH}_\pm(Q)
 \label{eq:graph-space-isomorphism}
\end{equation}
with equivalent graph norms.  This step is purely operator-theoretic and does
not use existence or uniqueness of the direct problem.

We may therefore define the generalized traces on \(\cH_{c,\pm}(Q)\) by
pullback from \(\widetilde{\cH}_\pm(Q)\).  Tilded trace symbols refer to the
graph spaces formed with \(\widetilde g\); in particular,
\(\widetilde\tau_1\) uses the outward unit normal
\(\nu_{\widetilde g}\), whereas \(\tau_1\) uses \(\nu_g\).  The Dirichlet and
temporal traces are metric-independent, but we retain tildes to indicate that
they act on the transformed function.  The resulting formulas, first valid
for \(u\in C^\infty(\overline Q)\) and then for all graph-space elements, are
\begin{align}
 \widetilde\tau_0(T_cu)&=c^\rho\tau_0u,
 \label{eq:app-dirichlet-transform}\\
 \widetilde r_j(T_cu)&=c^\rho r_ju,\qquad j\in\{0,T\},
 \label{eq:app-time-transform}\\
 \widetilde\tau_1(T_cu)
 &=c^{\rho+1/2}\tau_1u
 +\rho c^{-1}(\partial_{\nu_{\widetilde g}}c)
   \widetilde\tau_0(T_cu).
 \label{eq:app-neumann-transform}
\end{align}
The last identity is the product rule together with
\(\nu_g=c^{-1/2}\nu_{\widetilde g}\).  Its two terms lie in
\(H^{-3/2,-3/4}(\Sigma)\) because
\(H^{-1/2,-1/4}(\Sigma)\hookrightarrow H^{-3/2,-3/4}(\Sigma)\).

\begin{proposition}\label{prop:app-original-wellposedness}
Let \(q\in L^\infty(Q)\).  For every
\[
 F\in L^2(Q),\qquad
 f\in H^{-1/2,-1/4}(\Sigma),\qquad
 u_0\in H^{-1}(M),
\]
there exists a unique \(u\in\cH_{c,+}(Q)\) such that
\[
 \cL_{c,g,q}u=F\quad\text{in }Q,\qquad
 \tau_0u=f,\qquad r_0u=u_0
\]
in the transposition sense.  Moreover,
\begin{align}
 &\|u\|_{\cH_{c,+}(Q)}
 +\|r_Tu\|_{H^{-1}(M)}
 +\|\tau_1u\|_{H^{-3/2,-3/4}(\Sigma)}\notag\\
 &\qquad\leq C\bigl(
 \|F\|_{L^2(Q)}
 +\|f\|_{H^{-1/2,-1/4}(\Sigma)}
 +\|u_0\|_{H^{-1}(M)}\bigr).
 \label{eq:app-original-estimate}
\end{align}
The analogous backward statement holds for
\(\cL_{c,g,q}^*=-c^{-1}\partial_t-\Delta_g+\overline q\), with \(r_Tu\)
prescribed in place of \(r_0u\).  In particular, for \(F=0\) the partial
input--output map \(\Lambda_{c,g,q}^{U,V}\) defined in
\eqref{eq:partial-map} is well-defined and bounded.
\end{proposition}

\begin{proof}
By \eqref{eq:conformal-conjugation}, \(v=T_cu\) is required to solve
\begin{equation}
 (\partial_t-\Delta_{\widetilde g}+\widetilde q)v
 =c^{(n+2)/4}F,\qquad
 \widetilde\tau_0v=c^\rho f,\qquad
 \widetilde r_0v=c^\rho u_0.
 \label{eq:transformed-forced-problem}
\end{equation}
Proposition~\ref{prop:appendix-reduced-transposition} gives a unique solution
\(v\) of this reduced problem.  Put \(u=c^{-\rho}v\).
The graph-space isomorphism and
\cref{eq:app-dirichlet-transform,eq:app-time-transform} give the equation and
the prescribed traces for \(u\), and uniqueness follows by applying \(T_c\).
All norms involved are preserved up to fixed equivalence constants, which gives
\eqref{eq:app-original-estimate}.  The backward statement follows in the same
way.  Taking \(F=0\) and using the quotient norm on
\(H^{-3/2,-3/4}(V)\) proves the assertion for
\(\Lambda_{c,g,q}^{U,V}\).
\end{proof}

Finally, \cref{eq:app-time-transform,eq:app-neumann-transform} give exactly
\cref{eq:transformed-terminal,eq:transformed-neumann}.  Multiplication by
\(c^\rho\) is an isomorphism on \(H^{-1}(M)\),
\(H^{-1/2,-1/4}(\Sigma)\), and the restriction spaces, and it preserves the
support condition defining \(\cD_U\).  Thus the input transformation is
bijective and the additional Neumann term is known and potential-independent.
This proves \cref{prop:map-reduction}.

Together with the conformal discussion in Section~\ref{sec:preliminaries},
Proposition~\ref{prop:app-original-wellposedness} proves the well-posedness of
the original initial-boundary value problem \eqref{eq:original-ibvp} and of the
partial input--output map \eqref{eq:partial-map}, while the preceding trace
identities justify the reduction of these objects to the case \(c=1\).

\section{Boundary Carleman estimates}
\label{app:carleman}

Throughout this appendix, \(g=dx_1^2\oplus g_0\), and the weights are those in
\eqref{eq:parabolic-weights}.  In particular, the lower-case weight
\(\varphi_h=x_1+\beta^2t/h\) is the one convexified below, whereas
\(\Phi_h=\varphi_h/h\) is the exponent in the conjugated operator.

\begin{proposition}\label{prop:boundary-carleman}
Let \(q\in L^\infty(Q)\).  There exist \(C>0\) and \(h_0>0\) such that the
following estimates hold for \(0<h\leq h_0\).

If \(u\in C^\infty(\overline Q)\) satisfies
\[
 u|_\Sigma=0,\qquad u(0,\cdot)=0,
\]
then
\begin{align}
 &Ch^3\int_{\{\partial_\nu\varphi_h>0\}}
 |\partial_\nu\varphi_h|\,|\partial_\nu u|^2\,dS_gdt
 +Ch^2\bigl(\|u\|_{L^2(Q)}^2
 +\|h\nabla_gu\|_{L^2(Q)}^2\bigr)\notag\\
 &\quad\leq
 \|h^2e^{-\Phi_h}P_q(e^{\Phi_h}u)\|_{L^2(Q)}^2
 +h^2\|u(T,\cdot)\|_{L^2(M)}^2\notag\\
 &\qquad\quad
 +h^3\int_{\{\partial_\nu\varphi_h<0\}}
 |\partial_\nu\varphi_h|\,|\partial_\nu u|^2\,dS_gdt.
 \label{eq:carleman-forward}
\end{align}
If \(u\in C^\infty(\overline Q)\) satisfies
\[
 u|_\Sigma=0,\qquad u(T,\cdot)=0,
\]
then
\begin{align}
 &Ch^3\int_{\{\partial_\nu\varphi_h<0\}}
 |\partial_\nu\varphi_h|\,|\partial_\nu u|^2\,dS_gdt
 +Ch^2\bigl(\|u\|_{L^2(Q)}^2
 +\|h\nabla_gu\|_{L^2(Q)}^2\bigr)\notag\\
 &\quad\leq
 \|h^2e^{\Phi_h}P_q^*(e^{-\Phi_h}u)\|_{L^2(Q)}^2
 +h^2\|u(0,\cdot)\|_{L^2(M)}^2\notag\\
 &\qquad\quad
 +h^3\int_{\{\partial_\nu\varphi_h>0\}}
 |\partial_\nu\varphi_h|\,|\partial_\nu u|^2\,dS_gdt.
 \label{eq:carleman-adjoint}
\end{align}
The boundary sets are subsets of \(\Sigma\); the glancing set contributes
nothing.
\end{proposition}

\begin{proof}
We prove \eqref{eq:carleman-forward}.  The adjoint estimate follows by symmetry,
as explained at the end.  It is enough to prove the estimate for \(q=0\).
Writing \(P_0:=\partial_t-\Delta_g\), we have
\[
 \|h^2e^{-\Phi_h}P_0(e^{\Phi_h}u)\|_{L^2(Q)}
 \leq \|h^2e^{-\Phi_h}P_q(e^{\Phi_h}u)\|_{L^2(Q)}
 +h^2\|q\|_{L^\infty(Q)}\|u\|_{L^2(Q)}.
\]
After squaring, the last term is absorbed by the bulk
\(h^2\|u\|_{L^2(Q)}^2\) term for sufficiently small \(h\).  For \(q=0\), we
first prove the estimate for real-valued functions; applying it to the real and
imaginary parts gives the complex-valued result.  Thus all integrals in the
calculation below are real.

Fix \(\delta>0\) and convexify the lower-case weight by
\begin{equation}
 \varphi_{h,\delta}(t,x)
 :=\varphi_h(t,x)-\frac{h}{2\delta}x_1^2.
 \label{eq:convexified-weight}
\end{equation}
For a smooth \(v\) satisfying \(v|_\Sigma=0\) and \(v(0,\cdot)=0\), set
\[
 \cP_\delta v
 :=h^2e^{-\varphi_{h,\delta}/h}(\partial_t-\Delta_g)
 \bigl(e^{\varphi_{h,\delta}/h}v\bigr).
\]
Direct conjugation gives \(\cP_\delta=\cP_1+\cP_2\), where
\begin{align*}
 \cP_1&:=h^2\partial_t
 -2h\left(1-\frac h\delta x_1\right)\partial_{x_1}
 +\frac{4h^2}{\delta},\\
 \cP_2&:=-h^2\Delta_g+K(x_1),\\
 K(x_1)&:=-(1-\beta^2)-\frac{h^2}{\delta^2}x_1^2
 +\frac{2h}{\delta}x_1-\frac{3h^2}{\delta}.
\end{align*}
Indeed, before splitting, the zero-order coefficient is
\[
 -(1-\beta^2)-\frac{h^2}{\delta^2}x_1^2
 +\frac{2h}{\delta}x_1+\frac{h^2}{\delta}.
\]
Adding \(4h^2/\delta\) to \(\cP_1\) and subtracting it from \(\cP_2\)
produces the displayed formulas.

Write
\[
 (\cP_1v,\cP_2v)_{L^2(Q)}=\sum_{j=1}^6I_j,
\]
where
\begin{align*}
 I_1&=(h^2\partial_tv,-h^2\Delta_gv)_{L^2(Q)},&
 I_2&=(h^2\partial_tv,Kv)_{L^2(Q)},\\
 I_3&=\left(-2h\left(1-\frac h\delta x_1\right)
 \partial_{x_1}v,-h^2\Delta_gv\right)_{L^2(Q)},&
 I_4&=\left(-2h\left(1-\frac h\delta x_1\right)
 \partial_{x_1}v,Kv\right)_{L^2(Q)},\\
 I_5&=\left(\frac{4h^2}{\delta}v,-h^2\Delta_gv\right)_{L^2(Q)},&
 I_6&=\left(\frac{4h^2}{\delta}v,Kv\right)_{L^2(Q)}.
\end{align*}

We compute all six terms.  Since \(v=0\) on \(\Sigma\), also
\(\partial_tv=0\) there.  Spatial integration by parts followed by integration
in time gives
\begin{align}
 I_1
 &=-h^4\int_0^T\int_M\partial_tv\,\Delta_gv\,dV_gdt\notag\\
 &=h^4\int_0^T\int_M
 \langle\nabla_g\partial_tv,\nabla_gv\rangle_g\,dV_gdt\notag\\
 &=\frac{h^4}{2}\int_M|\nabla_gv(T,x)|_g^2\,dV_g.
 \label{eq:I1}
\end{align}
As \(K\) is independent of \(t\),
\begin{align}
 I_2
 &=\frac{h^2}{2}\int_MK(x_1)|v(T,x)|^2\,dV_g\notag\\
 &=-\frac{1-\beta^2}{2}h^2\|v(T,\cdot)\|_{L^2(M)}^2
 +O(h^3)\|v(T,\cdot)\|_{L^2(M)}^2.
 \label{eq:I2}
\end{align}

For \(I_3\), retain the lateral boundary term:
\begin{align*}
 I_3
 &=2h^3\int_Q\left(1-\frac h\delta x_1\right)
 \partial_{x_1}v\,\Delta_gv\,dV_gdt\\
 &=2h^3\int_\Sigma\left(1-\frac h\delta x_1\right)
 \partial_{x_1}v\,\partial_\nu v\,dS_gdt\\
 &\quad-2h^3\int_Q
 \left\langle\nabla_g\left[
 \left(1-\frac h\delta x_1\right)\partial_{x_1}v\right],
 \nabla_gv\right\rangle_g\,dV_gdt\\
 &=2h^3\int_\Sigma\left(1-\frac h\delta x_1\right)
 \partial_{x_1}v\,\partial_\nu v\,dS_gdt
 +\frac{2h^4}{\delta}\|\partial_{x_1}v\|_{L^2(Q)}^2\\
 &\quad-h^3\int_Q\left(1-\frac h\delta x_1\right)
 \partial_{x_1}|\nabla_gv|_g^2\,dV_gdt.
\end{align*}
The Dirichlet condition implies on \(\Sigma\) that
\[
 \nabla_gv=(\partial_\nu v)\nu,\qquad
 \partial_{x_1}v=(\partial_\nu\phi)\partial_\nu v.
\]
Integrating the last volume integral in the \(x_1\)-direction gives
\begin{align}
 I_3={}&h^3\int_\Sigma
 \left(1-\frac h\delta x_1\right)
 \partial_\nu\phi\,|\partial_\nu v|^2\,dS_gdt\notag\\
 &+\frac{2h^4}{\delta}\|\partial_{x_1}v\|_{L^2(Q)}^2
 -\frac{h^4}{\delta}\|\nabla_gv\|_{L^2(Q)}^2.
 \label{eq:I3}
\end{align}

The boundary term in \(I_4\) vanishes, and hence
\begin{align*}
 I_4
 &=-h\int_Q\left(1-\frac h\delta x_1\right)K(x_1)
 \partial_{x_1}(|v|^2)\,dV_gdt\\
 &=h\int_Q\partial_{x_1}\left[
 \left(1-\frac h\delta x_1\right)K(x_1)\right]
 |v|^2\,dV_gdt.
\end{align*}
A direct differentiation yields
\[
 h\partial_{x_1}\left[
 \left(1-\frac h\delta x_1\right)K(x_1)\right]
 =\frac{3-\beta^2}{\delta}h^2
 -\frac{6x_1}{\delta^2}h^3
 +\frac{3x_1^2}{\delta^3}h^4
 +\frac{3}{\delta^2}h^4.
\]
Therefore
\begin{align}
 I_4={}&\frac{3-\beta^2}{\delta}h^2\|v\|_{L^2(Q)}^2
 -\frac{6h^3}{\delta^2}\int_Qx_1|v|^2\,dV_gdt\notag\\
 &+\frac{3h^4}{\delta^3}\int_Qx_1^2|v|^2\,dV_gdt
 +\frac{3h^4}{\delta^2}\|v\|_{L^2(Q)}^2.
 \label{eq:I4}
\end{align}
The fifth term is
\begin{equation}
 I_5=\frac{4h^4}{\delta}\|\nabla_gv\|_{L^2(Q)}^2.
 \label{eq:I5}
\end{equation}
Finally,
\begin{align}
 I_6={}&-\frac{4(1-\beta^2)}{\delta}h^2\|v\|_{L^2(Q)}^2
 +\frac{8h^3}{\delta^2}\int_Qx_1|v|^2\,dV_gdt\notag\\
 &-\frac{4h^4}{\delta^3}\int_Qx_1^2|v|^2\,dV_gdt
 -\frac{12h^4}{\delta^2}\|v\|_{L^2(Q)}^2.
 \label{eq:I6}
\end{align}

The \(x_1\)-projection of \(M\) is bounded.  Combining
\cref{eq:I1,eq:I2,eq:I3,eq:I4,eq:I5,eq:I6}, and noting that the leading
\(L^2(Q)\) coefficient in \(I_4+I_6\) is
\[
 (3-\beta^2)-4(1-\beta^2)=3\beta^2-1,
\]
we obtain
\begin{align}
 (\cP_1v,\cP_2v)_{L^2(Q)}
 \geq{}&
 -\frac{1-\beta^2}{2}h^2\|v(T,\cdot)\|_{L^2(M)}^2
 +\frac{3\beta^2-1}{\delta}h^2\|v\|_{L^2(Q)}^2\notag\\
 &+\frac{3h^4}{\delta}\|\nabla_gv\|_{L^2(Q)}^2
 +h^3\int_\Sigma\left(1-\frac h\delta x_1\right)
 \partial_\nu\phi\,|\partial_\nu v|^2\,dS_gdt\notag\\
 &-Ch^3\|v\|_{L^2(Q)}^2
 -Ch^3\|v(T,\cdot)\|_{L^2(M)}^2.
 \label{eq:carleman-cross-term}
\end{align}
We discarded the nonnegative terms
\(h^4\|\nabla_gv(T,\cdot)\|_{L^2(M)}^2/2\) and
\(2h^4\|\partial_{x_1}v\|_{L^2(Q)}^2/\delta\).

The strict inequality \(3\beta^2-1>0\) absorbs the bulk error for sufficiently
small \(h\).  Compactness also allows us to arrange
\[
 \frac12\leq1-\frac h\delta x_1\leq\frac32\qquad\text{on }M.
\]
Use
\[
 \|\cP_\delta v\|_{L^2(Q)}^2
 =\|\cP_1v\|_{L^2(Q)}^2+\|\cP_2v\|_{L^2(Q)}^2
 +2(\cP_1v,\cP_2v)_{L^2(Q)}
\]
and move the unfavorable endpoint and boundary terms to the right.  This proves
the estimate with the convexified weight.

To remove the convexification, set
\[
 u=e^{-x_1^2/(2\delta)}v.
\]
Then
\[
 h^2e^{-\Phi_h}(\partial_t-\Delta_g)(e^{\Phi_h}u)
 =e^{-x_1^2/(2\delta)}\cP_\delta v.
\]
The multiplier and its first derivatives are bounded on \(M\), and the
multiplier is bounded away from zero.  Moreover,
\[
 \partial_\nu\varphi_{h,\delta}
 =\left(1-\frac h\delta x_1\right)\partial_\nu\varphi_h.
\]
Thus convexification preserves the positive and negative boundary parts and the
corresponding integrals are uniformly comparable.  This proves
\eqref{eq:carleman-forward} for \(q=0\); the bounded potential is absorbed as
at the start of the proof.

Finally, apply the forward estimate to \(v(s,x)=u(T-s,x)\), use the spatial
weight \(-x_1\), and replace \(q\) by \(\overline q(T-s,x)\).  Since
\[
 -\varphi_h(T-s,x)
 =-x_1+\frac{\beta^2s}{h}-\frac{\beta^2T}{h},
\]
the final term is a constant in the exponential and hence irrelevant.  Time
reversal changes the forward operator into the adjoint, while
\(\partial_\nu(-x_1)=-\partial_\nu\varphi_h\).  This proves
\eqref{eq:carleman-adjoint} without repeating the \(I_j\)-calculation.
\end{proof}

\section*{Generative AI statement}
During the preparation of this manuscript, the authors used OpenAI ChatGPT and
OpenAI Codex (GPT-5, accessed August 2026) to assist with drafting and revising
parts of the exposition and the \LaTeX{} source, with the aim of improving
clarity, organization, and typesetting consistency.  All mathematical arguments,
proofs, computations, and references were independently checked and revised by
the authors, who take full responsibility for the content.

\section*{Funding}
This work was supported by the National Natural Science Foundation of China
[grant number 12171178].

\section*{Disclosure statement}
The authors report that there are no competing interests to declare.

\section*{Data availability statement}
No data were used for the research described in this article.


\begin{thebibliography}{99}

\bibitem{AvdoninSeidman1995}
S.~Avdonin and T.~I. Seidman,
\newblock Identification of \(q(x)\) in \(u_t=\Delta u-qu\) from boundary
observations,
\newblock \emph{SIAM J. Control Optim.} \textbf{33} (1995), no.~4,
1247--1255. \url{https://doi.org/10.1137/S0363012993249729}.

\bibitem{BellassouedFraj2021}
M.~Bellassoued and O.~B. Fraj,
\newblock Stably determining time-dependent convection--diffusion coefficients
from a partial Dirichlet-to-Neumann map,
\newblock \emph{Inverse Problems} \textbf{37} (2021), no.~4, 045011.
\url{https://doi.org/10.1088/1361-6420/abe10d}.

\bibitem{BellassouedRassas2020}
M.~Bellassoued and I.~Rassas,
\newblock Stability estimate for an inverse problem of the
convection--diffusion equation,
\newblock \emph{J. Inverse Ill-Posed Probl.} \textbf{28} (2020), no.~1,
71--92. \url{https://doi.org/10.1515/jiip-2018-0072}.

\bibitem{BukhgeimUhlmann2002}
A.~L. Bukhgeim and G.~Uhlmann,
\newblock Recovering a potential from partial Cauchy data,
\newblock \emph{Comm. Partial Differential Equations} \textbf{27} (2002),
no.~3--4, 653--668. \url{https://doi.org/10.1081/PDE-120002868}.

\bibitem{CanutoKavian2001}
B.~Canuto and O.~Kavian,
\newblock Determining coefficients in a class of heat equations via boundary
measurements,
\newblock \emph{SIAM J. Math. Anal.} \textbf{32} (2001), no.~5, 963--986.
\url{https://doi.org/10.1137/S003614109936525X}.

\bibitem{ChoulliKian2013}
M.~Choulli and Y.~Kian,
\newblock Stability of the determination of a time-dependent coefficient in
parabolic equations,
\newblock \emph{Math. Control Relat. Fields} \textbf{3} (2013), no.~2,
143--160. \url{https://doi.org/10.3934/mcrf.2013.3.143}.

\bibitem{ChoulliKian2018}
M.~Choulli and Y.~Kian,
\newblock Logarithmic stability in determining the time-dependent zero order
coefficient in a parabolic equation from a partial Dirichlet-to-Neumann map.
Application to the determination of a nonlinear term,
\newblock \emph{J. Math. Pures Appl. (9)} \textbf{114} (2018), 235--261.
\url{https://doi.org/10.1016/j.matpur.2017.12.003}.

\bibitem{DosSantosFerreiraKenigSaloUhlmann2009}
D.~Dos Santos Ferreira, C.~E. Kenig, M.~Salo, and G.~Uhlmann,
\newblock Limiting Carleman weights and anisotropic inverse problems,
\newblock \emph{Invent. Math.} \textbf{178} (2009), no.~1, 119--171.
\url{https://doi.org/10.1007/s00222-009-0196-4}.

\bibitem{DosSantosFerreiraEtAl2016}
D.~Dos Santos Ferreira, Y.~Kurylev, M.~Lassas, and M.~Salo,
\newblock The Calder\'on problem in transversally anisotropic geometries,
\newblock \emph{J. Eur. Math. Soc. (JEMS)} \textbf{18} (2016), no.~11,
2579--2626. \url{https://doi.org/10.4171/JEMS/649}.

\bibitem{FanDuan2021}
J.~Fan and Z.~Duan,
\newblock Determining a potential of the parabolic equation from partial
boundary measurements,
\newblock \emph{Inverse Problems} \textbf{37} (2021), no.~9, 095001, 21~pp.
\url{https://doi.org/10.1088/1361-6420/ac156d}.

\bibitem{FeizmohammadiKianUhlmann2022}
A.~Feizmohammadi, Y.~Kian, and G.~Uhlmann,
\newblock An inverse problem for a quasilinear convection--diffusion equation,
\newblock \emph{Nonlinear Anal.} \textbf{222} (2022), 112921, 30~pp.
\url{https://doi.org/10.1016/j.na.2022.112921}.

\bibitem{FeizmohammadiKianUhlmann2024}
A.~Feizmohammadi, Y.~Kian, and G.~Uhlmann,
\newblock Partial data inverse problems for reaction--diffusion and heat
equations,
\newblock arXiv:2406.01387, 2024.

\bibitem{Isakov1991}
V.~Isakov,
\newblock Completeness of products of solutions and some inverse problems for
PDE,
\newblock \emph{J. Differential Equations} \textbf{92} (1991), no.~2,
305--316. \url{https://doi.org/10.1016/0022-0396(91)90051-A}.

\bibitem{KatchalovKurylevLassas2001}
A.~Katchalov, Y.~Kurylev, and M.~Lassas,
\newblock \emph{Inverse Boundary Spectral Problems},
\newblock Monographs and Surveys in Pure and Applied Mathematics, Vol.~123,
Chapman \& Hall/CRC, Boca Raton, FL, 2001.

\bibitem{KenigSalo2013}
C.~E. Kenig and M.~Salo,
\newblock The Calder\'on problem with partial data on manifolds and
applications,
\newblock \emph{Anal. PDE} \textbf{6} (2013), no.~8, 2003--2048.
\url{https://doi.org/10.2140/apde.2013.6.2003}.

\bibitem{KenigSjostrandUhlmann2007}
C.~E. Kenig, J.~Sj\"ostrand, and G.~Uhlmann,
\newblock The Calder\'on problem with partial data,
\newblock \emph{Ann. of Math. (2)} \textbf{165} (2007), no.~2, 567--591.
\url{https://doi.org/10.4007/annals.2007.165.567}.

\bibitem{KumarPurohit2025}
P.~Kumar and A.~Purohit,
\newblock Inverse boundary value problem for the convection--diffusion equation
with local data,
\newblock \emph{Applicable Analysis} \textbf{104} (2025), no.~11, 2195--2204.
\url{https://doi.org/10.1080/00036811.2025.2454385}.

\bibitem{KrupchykUhlmann2018}
K.~Krupchyk and G.~Uhlmann,
\newblock Inverse problems for advection diffusion equations in admissible
geometries,
\newblock \emph{Comm. Partial Differential Equations} \textbf{43} (2018),
no.~4, 585--615. \url{https://doi.org/10.1080/03605302.2018.1446163}.

\bibitem{LionsMagenes1972}
J.-L.~Lions and E.~Magenes,
\newblock \emph{Non-Homogeneous Boundary Value Problems and Applications},
Vol.~II,
\newblock Grundlehren der mathematischen Wissenschaften, Vol.~182,
Springer, Berlin--Heidelberg, 1972. \url{https://doi.org/10.1007/978-3-642-65217-2}.

\bibitem{LiuPurohit2026}
B.~Liu and A.~Purohit,
\newblock Recovery of time-dependent coefficients for the
convection--diffusion equation on conformally transversally anisotropic
manifolds from partial data,
\newblock arXiv:2608.04970, 2026.
\url{https://arxiv.org/abs/2608.04970}.

\bibitem{LiuSaksalaYan2024}
B.~Liu, T.~Saksala, and L.~Yan,
\newblock Partial data inverse problem for hyperbolic equation with
time-dependent damping coefficient and potential,
\newblock \emph{SIAM J. Math. Anal.} \textbf{56} (2024), no.~4, 5678--5722.
\url{https://doi.org/10.1137/23M1588676}.

\bibitem{LiuSaksalaYan2025}
B.~Liu, T.~Saksala, and L.~Yan,
\newblock Recovery of a time-dependent potential in hyperbolic equations on
conformally transversally anisotropic manifolds,
\newblock \emph{J. Spectr. Theory} \textbf{15} (2025), no.~1, 123--147.
\url{https://doi.org/10.4171/JST/547}.

\bibitem{MishraPurohitVashisth2025}
R.~K. Mishra, A.~Purohit, and M.~Vashisth,
\newblock Inverse problem for a time-dependent convection--diffusion equation
in admissible geometries,
\newblock \emph{Res. Math. Sci.} \textbf{12} (2025), article no.~75.
\url{https://doi.org/10.1007/s40687-025-00556-0}.

\bibitem{Ralston1982}
J.~Ralston,
\newblock Gaussian beams and the propagation of singularities,
\newblock in W.~Littman (ed.), \emph{Studies in Partial Differential
Equations}, MAA Studies in Mathematics, Vol.~23, Mathematical Association of
America, Washington, DC, 1982, 206--248.

\bibitem{SahooVashisth2020}
S.~K. Sahoo and M.~Vashisth,
\newblock A partial data inverse problem for the convection--diffusion
equation,
\newblock \emph{Inverse Probl. Imaging} \textbf{14} (2020), no.~1, 53--75.
\url{https://doi.org/10.3934/ipi.2019063}.

\bibitem{SaloUhlmann2011}
M.~Salo and G.~Uhlmann,
\newblock The attenuated ray transform on simple surfaces,
\newblock \emph{J. Differential Geom.} \textbf{88} (2011), no.~1, 161--187.
\url{https://doi.org/10.4310/jdg/1317758872}.

\bibitem{Zworski2012}
M.~Zworski,
\newblock \emph{Semiclassical Analysis},
\newblock Graduate Studies in Mathematics, vol.~138, American Mathematical
Society, Providence, RI, 2012.
\url{https://doi.org/10.1090/gsm/138}.

\end{thebibliography}
\end{document}